\documentclass[11pt,a4paper,reqno]{amsart} 
\usepackage{footnote}
\usepackage{amssymb}
\usepackage{latexsym}
\usepackage{amsmath}
\usepackage{mathrsfs}
\usepackage{amscd}
\usepackage[X2,T1]{fontenc}
\usepackage[applemac]{inputenc}
\usepackage{cite}

\usepackage{color}

\usepackage{calc}                   

\newcommand{\scal}[2]{\langle #1,#2\rangle}
\newcommand{\rr}[1]{\mathbf R^{#1}}
\newcommand{\rrstar}[1]{\mathbf R_*^{#1}}
\newcommand{\zz}[1]{\mathbf Z^{#1}}

\newcommand{\nm}[2]{\Vert #1\Vert _{#2}}
\newcommand{\NM}[2]{\left \Vert #1\right \Vert _{#2}}

\newcommand{\op}{\operatorname{Op}}

\newcommand{\sets}[2]{\{ \, #1\, ;\, #2\, \} }

\newcommand{\ep}{\varepsilon}
\newcommand{\fy}{\varphi}
\newcommand{\cdo}{\, \cdot \, }

\newcommand{\eabs}[1]{\langle #1\rangle}     

\newcommand{\vrum}{\vspace{0.1cm}}

\newcommand{\Trm}{\operatorname{Tr}}

\newcommand{\nn}[1]{{\mathbf N}^{#1}}

\newcommand{\maclB}{\mathcal B}

\newcommand{\maclL}{\mathcal L}
\newcommand{\maclM}{\mathcal M}

\newcommand{\maclS}{\mathcal S}
\newcommand{\maclW}{\mathcal W}
\newcommand{\mascB}{\mathscr B}

\newcommand{\mascF}{\mathscr F}
\newcommand{\mascI}{\mathscr I}
\newcommand{\mascN}{\mathscr N}
\newcommand{\mascP}{\mathscr P}
\newcommand{\mascS}{\mathscr S}

\newcommand{\GL}{\mathbf{M}}

\numberwithin{equation}{section}          

\newtheorem{thm}{Theorem}
\numberwithin{thm}{section}
\newtheorem*{tom}{\rubrik}
\newcommand{\rubrik}{}
\newtheorem{prop}[thm]{Proposition}
\newtheorem{cor}[thm]{Corollary}
\newtheorem{lemma}[thm]{Lemma}

\theoremstyle{definition}

\newtheorem{defn}[thm]{Definition}
\newtheorem{example}[thm]{Example}

\theoremstyle{remark}
\newtheorem{rem}[thm]{Remark}

\author{Joachim Toft}

\address{Department of Mathematics,
Linn{\ae}us University, V{\"a}xj{\"o}, Sweden}

\email{joachim.toft@lnu.se}

\author{Divyang G. Bhimani}

\address{Department of Mathematics,
Indian Institute of Science Education and Research, Pune, India}

\email{divyang.bhimani@iiserpune.ac.in}

\author{Ramesh Manna}

\address{School of Mathematical Sciences, National Institute
of Science Education and Research Bhubaneswar, Jatni, India}

\email{rameshmanna@niser.ac.in}

\title{Trace mappings on quasi-Banach modulation spaces and
applications to pseudo-differential operators of amplitude type}

\keywords{Modulation spaces, Gelfand-Shilov spaces, Wiener amalgam
spaces, Trace map, amplitude, pseudo-differential
operators}

\subjclass[2010]{35S05, 46F05, 42B35}

\begin{document}

\begin{abstract}
We deduce trace properties for modulation spaces (including
certain Wiener-amalgam spaces) of Gelfand-Shilov distributions.
We use these results to show that $\Psi$dos
with amplitudes in suitable modulation spaces, agree with normal type
$\Psi$dos whose symbols belong to (other) modulation spaces. In
particular we extend and improve the style of trace results for modulation spaces
in \cite{CorRod,Toft2} to include quasi-Banach modulation spaces.
We also apply our results to obtain Schatten-von 
Neumann and nuclearity properties for $\Psi$dos with amplidudes 
in modulation spaces, extending earlier work in
\cite{DeRuWa,Sjo,Toft2,Toft20}.
\end{abstract}

\maketitle

\par

\section{Introduction}\label{sec0}

\par

In the paper we deduce continuity properties of trace mappings when acting
on extended classes of modulation spaces. Thereafter we apply these
properties to get continuity and identification properties for amplitude type
pseudo-differential operators with amplitudes in modulation spaces. In contrast
to most of earlier approaches, e.{\,}g. \cite{FeHuWa,Toft2}
we allow the Lebesgue exponents of the involved modulation spaces to stay
in the full inverval $(0,\infty ]$. In particular, the modulation spaces in our
investigations are quasi-Banach spaces, which might not be Banach spaces.
Note that these Lebesgue exponents in \cite{Toft2}
should belong to the smaller interval $[1,\infty]$, while in \cite{FeHuWa}
they should belong to $(0,\infty )$ and thereby not allowed to attain $\infty$.
In particular, our investigations include trace mapping properties of the
modulation space $M^{\infty ,q}$, $q\in (0,1)$, while analogous investigations
in \cite{FeHuWa,Toft2} do not host these spaces.

\medspace

A trace map is an operator which reduce the dimension of the domain for
functions or distributions, by fixing some coordinates. For example, let
$f(x_1,x_2)$ be a function which depends on $x_1\in \Omega _1$ and
$x_2\in \Omega _2$, let $z\in \Omega _2$ be fixed. Then the
map $\Trm _{z}$ which takes $(x_1,x_2)\mapsto f(x_1,x_2)$ into
$x_1\mapsto f(x_1,z)$, i.{\,}e.,
$$
(\Trm _{z}f)(x_1) = f(x_1,z),
$$
can be considered as the archetype of trace mappings, provided
$x_1\mapsto f(x_1,z)$ makes sense as a function.
(See \cite{Ho1} and Section \ref{sec1} for notations.)

\par

Trace mappings
appear in natural ways in different kinds of problems, e.{\,}g. in boundary
value problems of partial differential equations. If
$P(t, x,D_t,D_x)$ is a partial differential operator, and $f$ and
$u_0$ are fixed functions or distributions, then
$$
\begin{cases}
P(t, x,D_t,D_x)u(t,x) =f(x,t), & t>0,
\\[1ex]
u(0,x) = u_0(x),
\end{cases}
$$
is an example of an initial value problem. It is expected that
the solution $u(t,x)$ should possess suitable continuity
and differentiability properties as well as the trace map
which takes $u(t,x)$ into $u_0(x)=u(0,x)$ is well-defined.
See e.{\,}g. \cite{Schu} and the references therein for more facts on this.

\par

Another example where
trace mappings appear naturally concerns pseudo-differential operators
of amplitude types.
For any amplitude $a\in \mascS (\rr {3d})$, the pseudo-differential operator
$\op (a)$ is the linear and continuous map from $\mascS (\rr d)$ to
$\mascS (\rr d)$, given by
\begin{align*}
(\op (a)f)(x) &= (2\pi )^{-d}\iint _{\rr {2d}} a(x,y,\zeta )
e^{i\scal {x-y}\zeta}f(y)\, dyd\zeta .
\intertext{Any such operator may in a unique way be expressed as a
pseudo-differential operator of standard
or Kohn-Nirenberg type with symbol in $\mascS (\rr {2d})$. That is, there is
a unique $a_0\in \mascS (\rr {2d})$ such that $\op (a)$ above is equal to
$a_0(x,D)=\op _0(a_0)$, where}
(\op _0(a_0)f)(x) &= (2\pi )^{-d}\iint _{\rr {2d}} a_0(x,\zeta )
e^{i\scal {x-y}\zeta}f(y)\, dyd\zeta .
\end{align*}
The drop of variable $y$ when passing from $a(x,y,\zeta )$ to $a_0(x,\zeta )$
implies that somewhere in the process, a trace map appear. Indeed,
$a_0$ is obtained by the formula
$$
a_0(x,\zeta ) = (e^{i\scal {D_\zeta}{D_y}}a)(x,y,\zeta ) \Big \vert _{y=x}
= (e^{i\scal {D_\zeta}{D_y}}a)(x,x+y,\zeta ) \Big \vert _{y=0}
$$
(see e.{\,}g. \cite[Section 18.2]{Ho1}).

\medspace

Trace mappings often possess convenient continuity
properties when acting on spaces of continuous functions. For example, the mappings
\begin{alignat}{2}
\Trm _{z} &: C^N(\rr {d_1+d_2}) & &\to C^N(\rr {d_1}),
\label{Eq:TraceMapClassical1}
\\[1ex]
\Trm _{z} &: \mascS (\rr {d_1+d_2}) & &\to \mascS (\rr {d_1}), 
\label{Eq:TraceMapClassical2}
\intertext{and}
\Trm _{z} &: \Sigma _1(\rr {d_1+d_2}) & &\to \Sigma _1(\rr {d_1}),
\label{Eq:TraceMapClassical3}
\end{alignat}
are continuous (and surjective).
More sensitive situations appear when the intended domains of $\Trm _{z}$ contain
elements with lack of continuity, or, more dreadful, host elements with heavy singularities.

\par

A classical example on such situations concerns trace mappings when acting on
Sobolev spaces $H^2_s(\rr {d_1+d_2})$. For $s>\frac {d_2}2$ one has that the
mappings \eqref{Eq:TraceMapClassical2} and \eqref{Eq:TraceMapClassical3}
are uniquely extendable to a continuous map
\begin{equation}\label{Eq:TraceSobolev}
\Trm _{z} : H^2_s(\rr {d_1+d_2}) \to H^2_{s_0}(\rr {d_1}),
\end{equation}
provided $s\ge s_0+\frac {d_2}2$ (see e.{\,}g. \cite{Ada}).

\par

The latter trace property was extended in \cite{Toft2} to modulation spaces which
include the Sobolev spaces above as special cases. We recall that the
modulation spaces $M^{p,q}_{(\omega )}(\rr d)$ and
$W^{p,q}_{(\omega )}(\rr d)$, introduced by Feichtinger in \cite{Fei0,Fei1}
and further developed in \cite{Fei6, FeiGro1,FeiGro2,FeiGro3,Gro2}
are the sets of (Gelfand-Shilov) distributions
whose short-time Fourier transforms belong to the weighted and mixed
Lebesgue spaces $L^{p,q}_{(\omega )}(\rr {2d})$ respectively
$L^{p,q}_{*,(\omega )}(\rr {2d})$.
Here $\omega$ is a weight function on phase (or time-frequency shift) space
and $p,q\in (0,\infty ]$. Note that $W^{p,q}_{(\omega )}(\rr d)$
is also an example on Wiener-amalgam spaces (cf. \cite{Fei6}).

\par

Since the introduction, modulation spaces have entered several
fields within mathematics and science, e.{\,}g.
the theory of pseudo-differential operators (see e.{\,}g.
\cite{Grochenig1b,Sjo,Tac,Toft15,BenOko,CorRod} and the references therein
for more recent progresses).

\par

It follows from
Theorem 3.2 in \cite{Toft2} that if $p,q\in [1,\infty ]$ and
\begin{equation}\label{Eq:TraceWeightCondIntro}
\omega _0(x,\xi )\eabs \eta ^t
\le C \omega (x,y,\xi ,\eta ),
\qquad
t\ge \frac {d_2}{q'}
\end{equation}
for some constant $C>0$, 
with the latter inequality strict when $q>1$, then
\eqref{Eq:TraceMapClassical3} is uniquely
extendable to a continuous map
\begin{align}
\Trm _{z} : M^{p,q}_{(\omega )}(\rr {d_1+d_2}) &\to M^{p,q}_{(\omega _0)}(\rr {d_1}),
\label{Eq:IntroTraceModulation}
\intertext{that is,}
\nm {\Trm _{z}f}{M^{p,q}_{(\omega _0)}} &\lesssim \nm f{M^{p,q}_{(\omega )}}.
\label{Eq:IntroTraceModulation2}
\end{align}
Here $\eabs x\equiv (1+|x|^2)^{\frac 12}$, $x\in \rr d$, as usual.
If
$$
p=q=2,
\quad
\omega (x,y,\xi ,\eta ) = \eabs {(\xi ,\eta )}^{s},
\quad
\omega _0(x,\xi ) = \eabs {\xi}^{s_0}
\quad \text{and}\quad
s\ge s_0+\frac {d_2}2,
$$
then \eqref{Eq:IntroTraceModulation}
agrees with \eqref{Eq:TraceSobolev}.

\par

There are other extensions of \eqref{Eq:TraceSobolev}.
In \cite[Theorem 3.3]{FeHuWa}, Feichtinger, Huang and Wang
use uniform-frequency decomposition techniques to establish trace properties
on $\alpha$-modulation spaces, and thereby achieve trace properties
on modulation and Besov spaces as special cases
(cf. \cite[Theorem 3.1]{FeHuWa}). 
In \cite{Schn}, Schneider deduce trace mapping results for Besov and
Triebel-Lizorkin spaces, allowing the involved Lebesgue exponents
to belong to the full interval $(0,\infty ]$.

\par

In Section \ref{sec2} we extend the
trace mapping result \eqref{Eq:IntroTraceModulation}
in the sense of relaxing the conditions of the involved
weight functions, allowing the Lebesgue exponents to belong to the
full interval $(0,\infty ]$, and complete \eqref{Eq:IntroTraceModulation}
and \eqref{Eq:IntroTraceModulation2} with trace maps for $W^{p,q}_{(\omega )}$
spaces. Especially we prove
$$
\Trm _{z} :
W^{p,q}_{(\omega )}(\rr {d_1+d_2})
\to
W^{p,q}_{(\omega _0)}(\rr {d_1})
$$
is continuous when $\omega$ and $\omega _0$ are the same as
for \eqref{Eq:IntroTraceModulation}. (See Theorems \ref{Thm:TraceMod}
and \ref{Thm:TraceMod2}.)
The involved weight functions should satisfy conditions of the form
\begin{align}
\omega (x+y) &\lesssim \omega (x)e^{r|y|}
\label{Eq:ModerateIntro}
\intertext{for some $r>0$, i.{\,}e., we permit general moderate weights
(cf. \cite{Gro3}). For example, for any $r\in \mathbf R$ and $\theta \in (0,1]$,
we allow the weights $(1+|x|)^r$ and $e^{r|x|^\theta}$.
Note that in \cite{FeHuWa,Toft2} the condition \eqref{Eq:ModerateIntro}
is replaced by}
\omega (x+y) &\lesssim \omega (x)\eabs y^r,
\label{Eq:PolModerateIntro}
\end{align}
for some $r\ge 0$, which is more restricted. In particular, weights like
$e^{r|x|^\theta}$ are not allowed in \cite{FeHuWa,Toft2}.

\par

The conditions on the weights have strong impact on the shape of
modulation spaces. For example, conditions of the form
\eqref{Eq:PolModerateIntro} imply that the modulation space
$M^{p,q}_{(\omega )}(\rr d)$ stay between $\mascS (\rr d)$ and
$\mascS '(\rr d)$, while conditions of the form \eqref{Eq:ModerateIntro}
imply that $M^{p,q}_{(\omega )}(\rr d)$ may contain the whole
$\mascS '(\rr d)$, or might be contained in
$\mascS (\rr d)$. In this respect, in contrast to \cite{FeHuWa,Toft2},
our trace mapping results
for modulation spaces in Section \ref{sec2} also include
ultra-distributions which are outside $\mascS '(\rr d)$.

\par

In particular we show that \eqref{Eq:IntroTraceModulation} holds
for such weights and $p,q\in (0,\infty ]$ after
\eqref{Eq:TraceWeightCondIntro} is relaxed into
\begin{equation}\tag*{(\ref{Eq:TraceWeightCondIntro})$'$}
\omega _0(x,\xi )\eabs \eta ^t
\le C \omega (x,y,\xi ,\eta )e^{r|y|},
\qquad
t\ge d_2\left (\max \left ( 1,\frac 1p, \frac 1q \right ) - \frac 1q \right ),
\end{equation}
for some $r\ge 0$, where the latter inequality in
\eqref{Eq:TraceWeightCondIntro}$'$ should be strict when $q>\min (p,1)$.
(See Theorem \ref{Thm:TraceMod}.)

\par

To reach such
general results for modulation spaces, we first use Gabor
expansions to deduce the quasi-norm estimate 
\eqref{Eq:IntroTraceModulation2} for elements
in the Gelfand-Shilov space $\Sigma _1(\rr d)$ (which is dense in
$\mascS(\rr d)$). Here convolution properties for weighted
discrete Lebesgue spaces (see \eqref{Eq:YoungSpec} and
\eqref{Eq:ConvExp} for details) play key roles to achieved the
desired estimates. 
Thereafter we deduce extensions of \eqref{Eq:IntroTraceModulation}
to the whole $M^{p,q}_{(\omega )}(\rr d)$ in the Banach space case,
$p,q\ge 1$ by applying Hahn-Banach's theorem. In similar way as in
\cite{Toft2}, some critical cases when $p=\infty$ are obtained by
using the \emph{narrow convergence}, a weaker form of convergence
compared to norm convergence, but sufficiently strong to guarantee
needed uniqueness properties (see \cite{Sjo,Toft2}).

\par

Extensions
to the general case, $p,q\in (0,\infty ]$ are then obtained by
applying suitable embedding results for modulation spaces,
exploiting the fact that $M^{p,q}_{(\omega )}(\rr d)$ increases
with $p$ and $q$. Finally the uniqueness of \eqref{Eq:IntroTraceModulation}
follows in the case $p,q<\infty$ by using the fact that $\Sigma _1(\rr d)$
is dense in $M^{p,q}_{(\omega )}(\rr d)$. For general $p$ and $q$,
the uniqueness of \eqref{Eq:IntroTraceModulation} is then reached by
embedding $M^{p,q}_{(\omega )}(\rr d)$ into other modulation spaces,
where uniqueness assertions hold. More precisely, we find suitable
weights $\omega _1$ and $\omega _{0,1}$ such that the diagram
\begin{equation}\label{Eq:commdiagram}
\begin{CD}
M^{p,q}_{(\omega )}(\rr d)    @>{\Trm _z}>> M^{p,q}_{(\omega _0)}(\rr {d_1})
\\
@V {I}VV        @VV{I}V\\
M^{1,1}_{(\omega _1)}(\rr d)     @>>{\Trm _z}>
M^{1,1}_{(\omega _{0,1})}(\rr {d_1})
\end{CD}
\end{equation}
commutes. Here $I$ denotes continuous inclusions. The uniqueness
of the first map $\Trm _z$ in \eqref{Eq:commdiagram} then follows
from the uniqueness of the second one in \eqref{Eq:commdiagram}.
Note that these arguments remind on those in \cite[Remark 3.1]{Schn},
where Schneider explain how trace mappings 
on Besov and Triebel-Lizorkin spaces can be extended to allow the
involved Lebesgue exponents to stay in the full interval $(0,\infty ]$.

\par

In contrast to Theorems 3.1 and 3.3 in \cite{FeHuWa}, our results do not include
any trace results for Besov or, more generally, those $\alpha$-modulation
spaces which are not modulation spaces. It is not obvious whether
the methods in Section \ref{sec2} are well-designed for such investigations.
On the other hand, the restrictions that the Lebesgue exponents
in Theorems 3.1 and 3.3 in \cite{FeHuWa} are not allowed to
attain $\infty$ is removed in Theorem \ref{Thm:TraceMod} in Section \ref{sec2}.
This restriction might also be removed by using the ideas in
\cite[Remark 3.1]{Schn} or behind \eqref{Eq:commdiagram}
in combination with embedding results in \cite{Grob,HaWa,ToWa}.
(See Remark \ref{Rem:TraceReaction}.)
%

\medspace

In Section \ref{sec3} we apply our trace mapping results in Section \ref{sec2}
to show that pseudo-differential operators with amplitudes in suitable
modulation spaces can be formulated as pseudo-differential operators of
Kohn-Nirenberg type with symbols in other modulation spaces. For example,
as a consequence of our investigations it follows that if $p,q\in (0,\infty ]$,
$\omega$ and $\omega _0$ are moderate weights such that
$$
\omega (x,x+z,\zeta +\eta ,\xi -\eta ,\eta ,z)
\asymp
\omega _0(x,\zeta ,\xi ,z)\eabs \eta ^t,\qquad t\ge \frac d{q'}
$$
with strict inequality when $q>1$, then
\begin{equation}\label{Eq:PseudoAmpKNIdent}
\op (M^{p,q}_{(\omega )}(\rr {3d})) = \op _0(M^{p,q}_{(\omega _0)}(\rr {2d})).
\end{equation}
(See Theorem \ref{Thm:AmpOpPsOpMod}.)
In particular, for any $a\in M^{p,q}_{(\omega )}(\rr {3d})$, there is a unique
$a_0\in M^{p,q}_{(\omega _0)}(\rr {2d})$ such that $\op (a)=\op _0(a_0)$.
On the other hand, for any $a_0\in M^{p,q}_{(\omega _0)}(\rr {2d})$,
there exists an $a\in M^{p,q}_{(\omega )}(\rr {3d})$ such that
$\op (a)=\op _0(a_0)$. (See Proposition \ref{Prop:AmpOpPsOpMod}.)
Hence, the map which takes
$a\in M^{p,q}_{(\omega )}(\rr {3d})$ into
$a_0\in M^{p,q}_{(\omega _0)}(\rr {2d})$ (which is uniquely defined)
is surjective.

\par

In order to deduce the surjectivity, a natural idea might be to choose
$a(x,y,\zeta )=a_0(x,\zeta )$ as candidate for $a$ above. On the other
hand, if $p<\infty$ and $a_0\in M^{p,q}_{(\omega _0)}(\rr {2d})$, then
$(x,y,\zeta )\mapsto a_0(x,\zeta )$ fails to belong to
$M^{p,q}_{(\omega )}(\rr {3d})$. Hence the choice
$a(x,y,\zeta )=a_0(x,\zeta )$ does not work in this case. Instead the choice
$$
a(x,y,\zeta ) = e^{-i\scal {D_\zeta}{D_y}}(a_0(x,\zeta )\fy (y-x))
$$
works when $\fy$ is a suitable function such that $\fy (0)=1$. 
(Cf. Remark \ref{Rem:TraceMod}, and Theorem \ref{Thm:AmpOpPsOpMod}
and its proof.)

\par

The identity \eqref{Eq:PseudoAmpKNIdent} (supplied by Theorem 
\ref{Thm:AmpOpPsOpMod}) leads to that any continuity or compactness
property for $\op _0(M^{p,q}_{(\omega _0)}(\rr {2d}))$
carry over to the class $\op (\maclM ^{p,q}_{(\omega )}(\rr {3d}))$.
Here $\maclM ^{p,q}_{(\omega )}(\rr {3d})$ is a modification
of $M^{p,q}(\rr {3d})$, obtained by suitable linear pullbacks
of the involved elements (see Section \ref{sec3} for strict definition). In
Theorem \ref{Thm:AmpOpModCont}$'$ in Section \ref{sec3} we use
\cite[Theorem 3.1]{Toft19} to achieve such continuity
and compactness properties
on modulation spaces, which are more general compared to
existing results in the literature (e.{\,}g. in \cite{Sjo,Toft2}).
For example, the following result is
an immediate consequence of Theorem \ref{Thm:AmpOpModCont}$'$
in Section \ref{sec3}, and is obtained by choosing
$p=\infty$ and $q\in (0,1]$ in that result. Here a common condition
for the involved weight function is
\begin{equation}\label{Eq:WeightPseudoCond}
\frac {\omega _2(x,\xi )}{\omega _1(z,\zeta )}
\lesssim
\omega (x,z,\zeta +\eta ,\xi -\zeta -\eta ,\eta ,z-x),
\qquad x,z,\xi ,\eta ,\zeta  \in \rr d.
\end{equation}

\par

%
%
%
%

\begin{thm}\label{Thm:AmpOpModCont}
Let $p\in (0,\infty ]$, $q\in (0,1]$,
$\omega \in \mascP _E(\rr {6d})$ and
$\omega _1,\omega _2\in \mascP _E(\rr {2d})$
be such that \eqref{Eq:WeightPseudoCond} holds.
If $a\in \maclM ^{\infty ,q}_{(\omega )}(\rr {3d})$,
then $\op (a)$ is continuous
map from $M^{p,q}_{(\omega _1)}(\rr d)$ to $M^{p,q}_{(\omega _2)}(\rr d)$.
\end{thm}

\par

In the case $q=1$, $\omega =1$ and $\omega _j=1$, $j=1,2$, Theorem
\ref{Thm:AmpOpModCont} is proved already in \cite{Sjo}.
For weights of polynomial type and Lebesgue exponents are obeyed
to stay in the restricted subinterval $[1,\infty ]$ of $(0,\infty ]$,
Theorem \ref{Thm:AmpOpModCont}$'$ is essentially proved in
\cite{Toft2}.

\par

Section \ref{sec3} also includes some investigations on
pseudo-differential operators of amplitude types with
symbols in modulation spaces of Wiener amalgam types.
For such operators we deduce continuity between
suitable modulation spaces and Wiener amalgam spaces
(see Theorem \ref{Thm:PseudoSymbWienerAm}).

\par

In Section \ref{sec3} we also combine Theorem \ref{Thm:AmpOpPsOpMod}
with suitable results in \cite{Toft19,Toft20} to obtain detailed compactness
results for pseudo-differential operators of amplitude types.
Especially the following Schatten-von Neumann and nuclearity results
are special cases of Theorem \ref{Thm:SchattenPseudo}$'$ and
Theorem \ref{Thm:NuclPseudo}$'$ in Section \ref{sec3}. (See
\cite[Theorem 3.4]{Toft19} and \cite[Theorem 4.2]{Toft20} for
related results.)

\par

\begin{thm}\label{Thm:SchattenPseudo}
Let $p,q\in (0,\infty ]$ be such that $q\le \min (p,p')$,
$\omega \in \mascP _E(\rr {6d})$ and
$\omega _1,\omega _2\in \mascP _E(\rr {2d})$ be such that
\eqref{Eq:WeightPseudoCond} holds.
If $a\in \maclM ^{p,q}_{(\omega )}(\rr {3d})$, then
$\op (a)\in \mascI _p(M^2_{(\omega _1)}(\rr d),M^2_{(\omega _2)}(\rr d))$,
and
$$
\nm {\op (a)}{\mascI _p(M^2_{(\omega _1)},M^2_{(\omega _2)})}
\lesssim \nm a {\maclM ^{p,q}_{(\omega )}},\quad
a\in \maclM ^{p,q}_{(\omega )}(\rr {3d}).
$$
\end{thm}

\par

\begin{thm}\label{Thm:NuclPseudo}
Let $p\in (0,1]$,
$\omega \in \mascP _E(\rr {6d})$ and
$\omega _1,\omega _2\in \mascP _E(\rr {2d})$ be such that
\eqref{Eq:WeightPseudoCond} holds.
If $a\in \maclM ^{p,p}_{(\omega )}(\rr {3d})$, then
$\op (a)\in \mascN _p(M^\infty _{(\omega _1)}(\rr d),M^p_{(\omega _2)}(\rr d))$,
and
$$
\nm {\op (a)}{\mascN _p(M^\infty _{(\omega _1)},M^p_{(\omega _2)})}
\lesssim \nm a {\maclM ^{p,p}_{(\omega )}},\quad
a\in \maclM ^{p,p}_{(\omega )}(\rr {3d}).
$$
\end{thm}

\par

The paper is organizes as follows. In Section \ref{sec1} we
recall some basic facts for Gelfand-Shilov, modulation spaces
and pseudo-differential operators. Thereafter we deduce trace
results for modulation spaces in Section \ref{sec2}. Finally
we apply these trace results in Section \ref{sec3} to transform
common continuity and compactness properties
for pseudo-differential operators of standard or Kohn-Nirenberg
types into related properties for pseudo-differential operators of
amplitude types.

%
%

\par

\section*{Acknowledgement}

\par

The first author was supported by Vetenskapsr{\aa}det
(Swedish Science Council), within the project 2019-04890.
The second author is thankful to
(DST/INSPIRE/04/2016/001507)
and the third author is thankful to
(DST/INSPIRE/04/2019/001914) for research grants.

\par

\section{Preliminaries}\label{sec1}

\par

In this section we recall some facts on Gelfand-Shilov spaces,
modulation spaces and pseudo-differential operators. After
introducing classes of weight functions and mixed norm spaces
of Lebesgue types, we recall some properties of the Gelfand-Shilov
space $\Sigma _1(\rr d)$ and its distribution space $\Sigma _1'(\rr d)$.
Thereafter we consider a class of modulation spaces which contains
the classical modulation spaces, introduced by Feichtinger in \cite{Fei1},
but are not that general as in the more general approach, given by
Feichtinger in \cite{Fei6}. In the last part we recall the definition of
pseudo-differential operators and present some basic facts.

\par

\subsection{Weight functions}\label{subsec1.1}
A \emph{weight} on $\rr d$ is a positive function
$\omega \in  L^\infty _{loc}(\rr d)$
such that $1/\omega \in  L^\infty _{loc}(\rr d)$. The weight $\omega$ on $\rr d$
is called \emph{moderate} if there is a positive locally bounded function
$v$ on $\rr d$ such that
\begin{equation}\label{eq:2}
\omega(x+y)\le C\omega(x)v(y),\quad x,y\in\rr{d},
\end{equation}
for some constant $C\ge 1$. If $\omega$ and $v$ are weights on $\rr d$ such
that \eqref{eq:2} holds, then $\omega$ is also called \emph{$v$-moderate}.
The set of all moderate weights on $\rr d$ is denoted by $\mascP _E(\rr d)$.

\par

We let $\mascP (\rr d)$ be the set of all weights of polynomial type. That is,
\eqref{eq:2} holds true for some (positive) polynomial $v$ on $\rr d$. For
$s\ge 1$ we also let
$\mascP _{E,s}(\rr d)$ ($\mascP _{E,s}^0(\rr d)$) be the set of all weights
on $\rr d$ such that for some $r>0$ (every $r>0$) there is a constant
$C>0$ such that
\begin{equation}\tag*{(\ref{eq:2})$'$}
\omega(x+y)\le C\omega(x)e^{r|y|^{\frac 1s}},\quad x,y\in\rr{d},
\end{equation}

\par

The weight $v$ on $\rr d$ is called \emph{submultiplicative},
if it is even and \eqref{eq:2}
holds for $\omega =v$. From now on, $v$ always denotes
a submultiplicative
weight if nothing else is stated. In particular,
if \eqref{eq:2} holds and $v$ is submultiplicative, then it follows
by straight-forward computations that
\begin{equation}\label{eq:2Next}
\begin{gathered}
\frac {\omega (x)}{v(y)} \lesssim \omega(x+y) \lesssim \omega(x)v(y),
\\[1ex]
\quad
v(x+y) \lesssim v(x)v(y)
\quad \text{and}\quad v(x)=v(-x),
\quad x,y\in\rr{d}.
\end{gathered}
\end{equation}
Here and in what follows we write
$A(\theta )\lesssim B(\theta )$, $\theta \in \Omega$,
if there is a constant $c>0$ such that $A(\theta )\le cB(\theta )$
for all $\theta \in \Omega$. We also set $A(\theta )\asymp B(\theta )$
when $A(\theta )\lesssim B(\theta )\lesssim A(\theta )$.

\par

If $\omega$ is a moderate weight on $\rr d$, then by \cite{Toft2,Toft10}
and above, there is a submultiplicative weight
$v$ on $\rr d$ such that \eqref{eq:2} and \eqref{eq:2Next}
hold. Moreover if $v$ is submultiplicative on $\rr d$, then
\begin{equation}\label{Eq:CondSubWeights}
1\lesssim v(x) \lesssim e^{r|x|}
\end{equation}
for some constant $r>0$ (cf. \cite{Gro3}). In particular, if
$\omega$ is moderate, then
\begin{equation}\label{Eq:ModWeightProp}
\omega (x+y)\lesssim \omega (x)e^{r|y|}
\quad \text{and}\quad
e^{-r|x|}\le \omega (x)\lesssim e^{r|x|},\quad
x,y\in \rr d
\end{equation}
for some $r>0$.

\par

\subsection{Mixed quasi-normed spaces of Lebesgue types}

\par

Let $p,q\in (0,\infty ]$, and let $\omega \in \mascP _E(\rr {2d})$.
Then $L^{p,q}_{(\omega )}(\rr {2d})$ and
$L^{p,q}_{*,(\omega )}(\rr {2d})$ consist
of all measurable functions $F$ on $\rr {2d}$ such that
\begin{alignat*}{3}
\nm {g_1}{L^q(\rr d)}&<\infty ,&
\quad &\text{where} &\quad
g_1(\xi ) &\equiv \nm {F(\cdo ,\xi )\omega (\cdo ,\xi )}{L^p(\rr d)}
\intertext{and}
\nm {g_2}{L^p(\rr d)}&<\infty ,&
\quad &\text{where} &\quad
g_2(x) &\equiv \nm {F(x,\cdo )\omega (x,\cdo )}{L^q(\rr d)},
\end{alignat*}
respectively.

\par

More generally, as in \cite{Toft19} we consider
general classes of mixed quasi-normed spaces of Lebesgue
types, parameterized by
$$
p =(p_1,\dots , p_d)\in (0,\infty ]^d,
\quad 
q =(q_1,\dots , q_d)\in (0,\infty ]^d,
$$
$\sigma \in \operatorname {S}_d$ and $\omega \in \mascP _E(\rr d)$.
Here $\operatorname {S}_d$ is the set of permutations
on $\{ 1,\dots ,d\}$. In fact, let $p \in
(0,\infty ]^d$, $\omega \in \mascP _E(\rr d)$, and let $\sigma
\in \operatorname {S}_d$. Moreover, let $\Omega _j\subseteq
\mathbf R$ be Borel-sets, $\mu _j$ be positive Borel
measures on $\Omega _j$, $j=1,\dots ,d$, and let
$\Omega =\Omega _1\times \cdots \times \Omega _d$
and $\mu = \mu _1\otimes \cdots \otimes \mu _d$.
For every measurable and complex-valued function $f$ on
$\Omega$, let
$g_{j,\omega ,\mu}$, $j=1,\dots ,d-1$, be defined inductively by
\begin{align*}
g_{0,\omega ,\mu}(x_1,\dots ,x_d)
&\equiv |f (x_{\sigma ^{-1}(1)},\dots ,x_{\sigma ^{-1}(d)})
\omega (x_{\sigma ^{-1}(1)},\dots ,x_{\sigma ^{-1}(d)})|,
\\[1ex]
g_{k,\omega ,\mu}(x_{k+1},\dots ,x_d) &\equiv \nm {g_{k-1,\omega ,\mu}(\cdo ,
x_{k+1},\dots ,x_d) }
{L^{p_k}(\mu _k)},
\quad k=1,\dots ,d-1 ,
\intertext{and let }
\nm f{L^{p}_{\sigma ,(\omega )}(\mu)} &\equiv
\nm {g_{d-1,\omega ,\mu}}{L^{p_d}(\mu _d)}.
\end{align*}
The mixed quasi-norm space $L^{p}_{\sigma ,(\omega )}(\mu)$ of
Lebesgue type is defined as the set of all $\mu$-measurable functions
$f$ such that $\nm f{L^{p }_{\sigma ,(\omega )}(\mu )}<\infty$.

\par

In the sequel we have $\Omega =\rr d$ and $d\mu = dx$, or
$\Omega =\Lambda$ and $\mu (j)=1$ when $j \in \Lambda$, where
\begin{equation}\label{LambdaDef}
\begin{aligned}
\Lambda &= \Lambda _{[\theta ]} = T_\theta \zz d \equiv
\sets {(\theta _1j_1,\dots ,\theta _dj_d)}{(j_1,\dots ,j_d)\in \zz d} ,
\\[1ex]
\theta &=(\theta _1,\dots ,\theta _d)\in \rrstar d,
\qquad \mathbf R_*=\mathbf R\setminus 0,
\end{aligned}
\end{equation}
and $T_\theta$ denotes the diagonal matrix with diagonal elements
$\theta _1,\dots ,\theta _d$. In the former case we set
$L^{p}_{\sigma ,(\omega )}(\mu)=L^{p}_{\sigma ,(\omega )}=
L^{p}_{\sigma ,(\omega )}(\rr d)$, and in the latter
case we set $L^{p}_{\sigma ,(\omega )}(\mu) =
\ell ^{p}_{\sigma ,(\omega )}(\Lambda )$.

\par

For convenience we set
$$
L^{p,q}_{(\omega )} = L^{p,q}_{\sigma _1,(\omega )}
\qquad \text{and}\qquad 
L^{p,q}_{*,(\omega )} = L^{q,p}_{\sigma _2,(\omega )}
$$
when $p,q\in (0,\infty ]^d$, $\omega \in \mascP _E(\rr {2d})$,
$\sigma _1\in \operatorname {S}_{2d}$ is the identity
map and $\sigma _2 \in \operatorname {S}_{2d}$ is given by
\begin{equation}\label{Eq:sigma2Def}
\sigma _2(j)=j+d
\quad \text{and}\quad
\sigma _2(j+d)=j,
\qquad j=1,\dots ,d.
\end{equation}

\par

Later on it is common that $p$ above is split in some ways. If
$d_j\in \mathbf N$, $j=1,\dots ,n$,
$$
p_j = (p_{j,1},\dots ,p_{j,d_j})\in (0,\infty ]^{d_j},\ j\in \{ 1,\dots ,n\} ,
$$
and
$$
p=(p_1,\dots ,p_n) = (p_{1,1},\dots ,
p_{1,d_1},\dots , p_{n,1}, \dots ,p_{n,d_n}),
$$
then set
\begin{align}
L^{p_1,\dots ,p_n}_{(\omega )}(\rr {d_1+\cdots +d_n}) &=
L^p_{(\omega )}(\rr {d_1+\cdots +d_n})
\label{Eq:MixedLebSpace}
\intertext{and}
\nm f{L^{p_1,\dots ,p_n}_{(\omega )}}
&=
\nm f{L^{p_1,\dots ,p_n}_{(\omega )}(\rr {d_1+\cdots +d_n})}
=
\nm f{L^p_{(\omega )}},
\intertext{when $f$ is complex-valued and measurable
on $\rr {d_1+\cdots +d_n}$. If in addition
$d_1=\cdots =d_n=d$ for some $d\ge 1$, then the space
in \eqref{Eq:MixedLebSpace} becomes}
L^p_{(\omega )}(\rr {nd}) &= L^{p_1,\dots ,p_n}_{(\omega )}(\rr {nd})
\tag*{(\ref{Eq:MixedLebSpace})$'$}
\end{align}
with $p_j\in (0,\infty ]^d$ for every $j=1,\dots ,n$.

\par


\par

\subsection{Gelfand-Shilov spaces}\label{subsec1.2}

\par

Let $0<h,s\in \mathbf R$ be fixed. Then $\mathcal S_{s;h}(\rr d)$ consists of
all $f\in C^\infty (\rr d)$ such that
\begin{equation*}
\nm f{\mathcal S_{s;h}}\equiv \sup \frac {|x^\beta \partial ^\alpha
f(x)|}{h^{|\alpha | + |\beta |}\alpha !^s\, \beta !^s}
\end{equation*}
is finite. Here the supremum is taken over all $\alpha ,\beta \in
\mathbf N^d$ and $x\in \rr d$.

\par

Obviously $\mathcal S_{s;h}\hookrightarrow
\mathscr S$ is a Banach space which increases with $h$ and $s$. Here and
in what follows we use the notation $A\hookrightarrow B$ when the topological
spaces $A$ and $B$ satisfy $A\subseteq B$ with continuous inclusion.

\par

The \emph{Gelfand-Shilov spaces} $\mathcal S_{s}(\rr d)$ and
$\Sigma _{s}(\rr d)$ are the inductive and projective limits respectively
of $\mathcal S_{s;h}(\rr d)$ with respect to $h>0$. This implies that
\begin{equation}\label{GSspacecond1}
\mathcal S_{s}(\rr d) = \bigcup _{h>0}\mathcal S_{s;h}(\rr d)
\quad \text{and}\quad \Sigma _{s}(\rr d) =\bigcap _{h>0}\mathcal S_{s;h}(\rr d),
\end{equation}
and that the topology for $\mathcal S_{s}(\rr d)$ is the
strongest possible one such
that the inclusion map from $\mathcal S_{s;h}(\rr d)$ to $\mathcal S_{s}(\rr d)$
is continuous, for every choice of $h>0$. The space $\Sigma _s(\rr d)$ is a
Fr{\'e}chet space with semi norms
$\nm \cdo{\mathcal S_{s;h}}$, $h>0$. Moreover, $\Sigma _s(\rr d)\neq \{ 0\}$,
if and only if $s>1/2$, and $\maclS _s(\rr d)\neq \{ 0\}$
if and only if $s\ge 1/2$
(cf. \cite{GeSh,Pil}).

\medspace

Let $\maclS _{s;h}'(\rr d)$ be the dual of $\maclS _{s;h}(\rr d)$.
Then the \emph{Gelfand-Shilov distribution spaces} $\mathcal S_{s}'(\rr d)$
and $\Sigma _s'(\rr d)$ are the projective and inductive limit
respectively of $\mathcal S_{s;h}'(\rr d)$ with respect to $h>0$.  This means that
\begin{equation}\tag*{(\ref{GSspacecond1})$'$}
\mathcal S_s'(\rr d) = \bigcap _{h>0}\mathcal S_{s;h}'(\rr d)\quad
\text{and}\quad \Sigma _s'(\rr d) =\bigcup _{h>0} \mathcal S_{s;h}'(\rr d).
\end{equation}
We remark that $\mathcal S_s'(\rr d)$
is the (strong) dual of $\mathcal S_s(\rr d)$ when $s\ge \frac 12$,
and $\Sigma _s'(\rr d)$ is the (strong) dual of $\Sigma _s(\rr d)$
when $s>\frac 12$ (cf. \cite{Pil}).
We also remark that the form
$(\cdo ,\cdo )_{L^2}=(\cdo ,\cdo )_{L^2(\rr d)}$
restricted to $\maclS _s(\rr d)\times \maclS _s(\rr d)$
($\Sigma _s(\rr d)\times \Sigma _s(\rr d)$) is uniquely extendable to
a continuous map from  $\maclS _s'(\rr d)\times \maclS _s(\rr d)$
($\Sigma _s'(\rr d)\times \Sigma _s(\rr d)$) to $\mathbf C$.

\par

We have
\begin{multline}\label{GSembeddings}
\maclS _{1/2}(\rr d) \hookrightarrow \Sigma _{s_1} (\rr d) \hookrightarrow 
\maclS _{s_1}(\rr d) \hookrightarrow \Sigma _{s_2}(\rr d) \hookrightarrow
\mascS (\rr d)
\\[1ex]
\hookrightarrow \mascS '(\rr d)
\hookrightarrow
\Sigma _{s_2}' (\rr d) \hookrightarrow  \maclS _{s_1}'(\rr d)
\hookrightarrow  \Sigma _{s_1}'(\rr d) \hookrightarrow \maclS _{1/2}'(\rr d),
\qquad \frac 12 < s_1<s_2,
\end{multline}
with dense embeddings.

\par

The Gelfand-Shilov spaces are invariant under several basic transformations.
For example they are invariant under translations, dilations
and under (partial) Fourier transformations. We also note that the map
$(f_1,f_2)\mapsto f_1\otimes f_2$ is continuous from $\maclS _s(\rr {d_1})
\times \maclS _s(\rr {d_2})$ to $\maclS _s(\rr {d_1+d_2})$, and similarly when
each $\maclS _s$ are replaced by $\Sigma _s$, $\maclS _s'$ or by
$\Sigma _s'$. (See also \cite{Toft22}.)

\par

We let $\mathscr F$ be the Fourier transform which takes the form
$$
(\mathscr Ff)(\xi )= \widehat f(\xi ) \equiv (2\pi )^{-\frac d2}\int _{\rr
{d}} f(x)e^{-i\scal  x\xi }\, dx
$$
when $f\in L^1(\rr d)$. Here $\scal \cdo \cdo$ denotes the usual scalar product
on $\rr d$. The map $\mathscr F$ extends 
uniquely to homeomorphisms on $\mathscr S'(\rr d)$, $\mathcal S_s'(\rr d)$
and $\Sigma _s'(\rr d)$, and restricts to 
homeomorphisms on $\mathscr S(\rr d)$,
$\mathcal S_s(\rr d)$ and $\Sigma _s(\rr d)$, 
and to a unitary operator on $L^2(\rr d)$.

\par

There are several characterizations of Gelfand-Shilov spaces and their 
distribution spaces (cf. \cite{ChuChuKim,Eij,Toft18} and the references therein).
For example, it follows from \cite{ChuChuKim,Eij} that $f\in \maclS _s (\rr d)$
($f\in \Sigma _s (\rr d)$), if and only if
\begin{equation}\label{GSexpcond}
|f(x)|\lesssim e^{-r|x|^{\frac 1s}}\quad \text{and}\quad |\widehat f (\xi )|
\lesssim
e^{-r |\xi |^{\frac 1s}} 
\end{equation}
is true for some $r>0$ (for every $r>0$).

\par

Gelfand-Shilov spaces and their distribution spaces can also be
characterized by estimates on their short-time Fourier transforms.
Let $\phi \in \maclS _s(\rr d)$ ($\phi \in \Sigma _s(\rr d)$) be fixed.
Then the short-time Fourier transform of $f\in \maclS _s'(\rr d)$
(of $f\in \Sigma _s'(\rr d)$) with respect to $\phi$ is defined by
\begin{align}
(V_\phi f)(x,\xi ) &\equiv (2\pi )^{-\frac d2}(f,\phi (\cdo -x)e^{i\scal \cdo \xi})_{L^2}.
\label{Eq:STFTDef}
\intertext{We observe that}
(V_\phi f)(x,\xi ) &= \mascF (f\cdot \overline {\phi (\cdo -x)})(\xi )
\tag*{(\ref{Eq:STFTDef})$'$}
\intertext{(cf. \cite{Toft22}). If in addition $f\in L^p(\rr d)$ for some $p\in [1,\infty ]$, then}
(V_\phi f)(x,\xi ) &= (2\pi )^{-\frac d2}\int _{\rr d}f(y)\overline {\phi (y-x)}e^{-i\scal y\xi}\, dy.
\tag*{(\ref{Eq:STFTDef})$''$}
\end{align}

\par

In the next lemma we present characterizations of
Gelfand-Shilov spaces and their distribution spaces
in terms of estimates on the short-time Fourier transforms of the
involved elements. The proof is omitted, since
the first part follows from \cite{GroZim}, and the second part from
\cite{Toft10,Toft18}.

\par

\begin{lemma}\label{Lemma:GSFourierest}
Let $p\in [1,\infty ]$, $f\in \mathcal S'_{1/2}(\rr d)$, $s\ge \frac 12$,
$\phi \in \maclS _s(\rr d)\setminus 0$ ($\phi \in \Sigma _s(\rr d)\setminus 0$)
and
$$
v_r(x,\xi ) = e^{r(|x|^{\frac 1s}+|\xi |^{\frac 1s})},\qquad r\ge 0.
$$
Then the following is true:
\begin{enumerate}
\item[(1)] $f\in \mathcal S_s(\rr d)$ ($f\in \Sigma _s(\rr d)$), if and only if
\begin{equation}\label{Eq:GSExpCond1}
\nm {V_\phi f\cdot v_r}{L^p}< \infty
\end{equation}
for some $r>0$ (for every $r>0$);

\vrum

\item[(2)] $f\in \mathcal S_s'(\rr d)$ ($f\in \Sigma _s'(\rr d)$), if and only if
\begin{equation}\label{Eq:GSExpCond2}
\nm {V_\phi f/ v_r}{L^p}< \infty
\end{equation}
for every $r>0$ (for some $r>0$).
\end{enumerate}
\end{lemma}

\par

%

\begin{rem}\label{Rem:GSFourierest}
Evidently, if $d_j\in \mathbf N$ are such that $d=d_1+d_2+d_3\ge 1$, $j=1,2,3$,
$x_j\in \rr {d_j}$,  $x=(x_1,x_2,x_3)\in \rr d$ and
$\xi =(\xi _1,\xi _2,\xi _3)\in \rr d$, then we may replace $v_r$ in
\eqref{Eq:GSExpCond1} and \eqref{Eq:GSExpCond2} by
$$
v_r(x,\xi ) =
e^{r(|x|^{\frac 1s} + |\xi _1|^{\frac 1s} + |\xi _2|^{\frac 1s} + |\xi _3|^{\frac 1s})},
$$
in order for Lemma \ref{Lemma:GSFourierest} should hold true. In particular
it follows that if $\phi \in \Sigma _1(\rr d)$ and $f\in \maclS _{1/2}'(\rr d)$,
then $f\in \Sigma _1'(\rr d)$, if and only if
\begin{equation}\tag*{(\ref{Eq:GSExpCond2})$'$}
(x,\xi )\mapsto V_\phi f(x,\xi )e^{-r(|x|+|\xi _1|+|\xi _2|+|\xi _3|)}
\end{equation}
belongs to $L^1(\rr d)$ for some $r>0$.
\end{rem}

\par

\subsection{Modulation spaces}\label{subsec1.4}

\par

As in \cite{Toft19} we need a broader family of modulation spaces
than what is presented
in \cite{Fei1,MolPfa}. See also \cite{Fei6}
for even more general modulation spaces.

\par

\begin{defn}\label{Def:GenModSpaces}
Let $p\in (0,\infty ]^{2d}$, $\sigma \in \operatorname{S}_{2d}$,
$\omega \in \mascP _E(\rr {2d})$ and
$\phi \in \Sigma _1(\rr d)\setminus 0$. Then
the modulation space $M^p_{\sigma ,(\omega )}(\rr d)$
consists of all $f\in \Sigma _1'(\rr d)$ such that
$$
\nm f{M^p_{\sigma ,(\omega )}}
\equiv
\nm {V_\phi f}{L^p_{\sigma ,(\omega )}}
$$
is finite.
\end{defn}

\par

\begin{rem}\label{Rem:ModSpaces}
Let $p,q\in (0,\infty ]^d$, $\sigma _1\in 
\operatorname{S}_{2d}$ be the identity map,
$\sigma _2\in \operatorname{S}_{2d}$
be given by \eqref{Eq:sigma2Def} and
$\omega \in \mascP _E(\rr {2d})$. Then
set
$$
M^{p,q}_{(\omega )} \equiv M^{(p,q)}_{\sigma _1,(\omega)}
\quad \text{and}\quad
W^{p,q}_{(\omega )} \equiv M^{(q,p)}_{\sigma _2,(\omega)}.
$$
We observe that $M^{p,q}_{(\omega )}(\rr d)$ is a slight
generalization of the standard modulation spaces, introduced
in \cite{Fei1}, and $W^{p,q}_{(\omega )}(\rr d)$ is a Wiener
amalgam space, considered in \cite{Fei0}.

\par

In some situations later on one has
that $p$ and $q$ are given by
\begin{alignat*}{2}
p &= & (p_1,\dots ,p_n) &= (p_{1,1},\dots ,p_{1,d_1},
\dots ,p_{n,1},\dots ,p_{n,d_n}),
\\[1ex]
q &= & (q_1,\dots ,q_n) &= (q_{1,1},\dots ,q_{1,d_1},\dots 
,q_{n,1},\dots ,q_{n,d_n}),
\end{alignat*}
for some $d_j\in \mathbf N$
$j=1,\dots ,n$, where
$$
p_j=(p_{j,1},\dots ,p_{j,d_j})\in (0,\infty ]^{d_j},
\quad
q_j=(q_{j,1},\dots ,q_{j,d_j})\in (0,\infty ]^{d_j}.
$$
It follows that
$$
M^{p,q}_{(\omega )}(\rr d)
=
M^{p_1,\dots ,p_n,q_1,\dots ,q_n}_{(\omega )}(\rr d)
=
M^{p_1,\dots ,p_n,q_1, \dots ,q_n}_{(\omega )}
(\rr {d_1+\cdots +d_n})
$$
and
$$
W^{p,q}_{(\omega )}(\rr d)
=
W^{p_1,\dots ,p_n,q_1,\dots ,q_n}_{(\omega )}(\rr d)
=
W^{p_1,\dots ,p_n,q_1, \dots ,q_n}_{(\omega )}
(\rr {d_1+\cdots +d_n})
$$
consist of all $f\in \Sigma _1'(\rr d)$ such that
$$
\nm f{M^{p,q}_{(\omega )}}
\equiv \nm {V_\phi f}{L^{p,q}_{(\omega )}}
\quad \text{respectively}\quad
\nm f{W^{p,q}_{(\omega )}}
\equiv \nm {V_\phi f}{L^{p,q}_{*,(\omega )}}
$$
are finite. Here $d=d_1+\dots +d_n$.
\end{rem}

\par

In our situations it is common that $n=3$ in Remark \ref{Rem:ModSpaces},
i.{\,}e.,
$$
M^{p,q}_{(\omega )}(\rr {d_1+d_2+d_3})
=
M^{p_1,p_2,p_3,q_1,q_2,q_3}_{(\omega )}(\rr {d_1+d_2+d_3}).
$$
It is also common that $d_1=d_2=d_3=d\ge 1$. In this situation
we write
$$
M^{p,q}_{(\omega )}(\rr {d+d+d}) = M^{p,q}_{(\omega )}(\rr {3d})
=
M^{p_1,p_2,p_3,q_1,q_2,q_3}_{(\omega )}(\rr {3d})
\phantom .
$$
and
$$
W^{p,q}_{(\omega )}(\rr {d+d+d}) = W^{p,q}_{(\omega )}(\rr {3d})
=
W^{p_1,p_2,p_3,q_1,q_2,q_3}_{(\omega )}(\rr {3d}).
$$

\par

In what follows, the conjugate exponent $p' \in (0,\infty ]$ of
$p \in (0,\infty ]$ is defined by
$$
p'
=
\begin{cases}
1 & \text{when}\ p=\infty
\\[1ex]
\frac p{p-1}  & \text{when}\ 1<p<\infty
\\[1ex]
\infty  & \text{when}\ p\le 1.
\end{cases}
$$
For any $p=(p_1,\dots ,p_n)\in (0,\infty ]^n$ and
$q=(q_1,\dots ,q_n)\in (0,\infty ]^n$ we set 
$$
p^{-1}=\frac 1p = (p_1^{-1},\dots ,p_n^{-1}),
\quad
p'=(p_1',\dots ,p_n'),
\quad
p+q=(p_1+q_1,\dots ,p_n+q_n)
$$
and
$$
p\le q\quad (p<q)
\qquad \text{when}\qquad
p_j\le q_j \quad (p_j<q_j),\ j=1,\dots ,n.
$$
For $r\in (0,\infty ]$ we also set
$$
p=r,
\quad
p\le r
\quad \text{respectively}\quad
p<r
$$
when $p_j=r$, $p_j\le r$ respectively $p_j<r$ for
every $j=1,\dots ,n$.

\par

In the following proposition we list some basic properties of modulation spaces.
We omit the proof since the result follows by straight-forward
generalizations of the analysis in \cite{GaSa,Gro2,Toft13} (see also
\cite{FeiGro1,FeiGro2,FeiGro3,Rau1,Rau2}).

\par

\begin{prop}\label{Prop:ModSpacesBasicProp}
Let $p,q\in (0,\infty ]^n$, $r\in (0,1]$ be such that $r\le \min (p,q)$,
$d_j\in \mathbf N$, $j=1,\dots ,n$,
$d=d_1+\cdots +d_n$, $\omega ,v\in \mascP _E(\rr {2d})$
be such that $\omega$ is $v$-moderate, and let
$\phi \in \Sigma _1(\rr d)\setminus 0$.
\begin{enumerate}
\item $\Sigma _1(\rr d)\subseteq M^{p,q}_{(\omega )}(\rr d)
\subseteq \Sigma _1'(\rr d)$. If in addition $\max (p,q)<\infty$, then
$\Sigma _1(\rr d)$ is dense in $M^{p,q}_{(\omega )}(\rr d)$;

\vrum

\item the definitions of $M^{p,q}_{(\omega )}(\rr d)$
is independent of the choices of $\phi \in M^1_{(v)}(\rr d)\setminus 0$,
and different choices of $\phi$ give rise to equivalent quasi-norms;

\vrum

\item $M^{p,q}_{(\omega )}(\rr d)$ is a quasi-Banach space which
increase with $p$ and $q$, and decrease with
$\omega$. If in addition $\min (p,q)\ge 1$, then
$M^{p,q}_{(\omega )}(\rr d)$ is a Banach space;

\vrum

\item The $L^2(\rr d)$ scalar product,
$(\cdo ,\cdo )_{L^2(\rr d)}$, on $\Sigma _1(\rr d)\times \Sigma _1(\rr d)$ is
uniquely extendable to a duality between $M^{p,q}_{(\omega)}(\rr d)$
and $M^{p',q'}_{(1/\omega)}(\rr d)$.

\par

If in addition $p,q<(\infty ,\dots ,\infty )$, then
the dual of $M^{p,q}_{(\omega)}(\rr d)$ can be identified with
$M^{p',q'}_{(1/\omega)}(\rr d)$, through $(\cdo ,\cdo )_{L^2(\rr d)}$;

\vrum

\item Let $\omega _0(x,\xi )= \omega (-\xi ,x)$. Then $\mascF$ on $\Sigma _1'(\rr d)$
restricts to a homeomorphism from $M^{p,q}_{(\omega )}(\rr d)$ to
$W^{q,p}_{(\omega _0)}(\rr d)$.
\end{enumerate}

\par

Similar facts hold true with $W^{p,q}_{(\omega)}$ spaces in place
of corresponding $M^{p,q}_{(\omega)}$ at each occurrence.
\end{prop}

\par

We set $M^{p,q}=M^{p,q}_{(\omega )}$
and $W^{p,q}=W^{p,q}_{(\omega )}$ when $\omega =1$ and
$p,q\in (0,\infty ]^n$. We observe that
$$
M^{p_0,q_0} = M^{p,q},
\quad
W^{p_0,q_0} = W^{p,q}
\quad \text{and}\quad
M^{p_0}=M^{p_0,p_0}
$$
when
$$
p=(p_0,\dots ,p_0)\in (0,\infty]^n
\quad \text{and}\quad
q=(q_0,\dots ,q_0)\in (0,\infty ]^n.
$$

\par

An important property for modulation spaces is that it is possible to
discretize them in terms of Gabor expansions. The following result is
based on \cite[Theorem S]{Gro1} and the Gabor analysis in \cite{GaSa}.
(See also Proposition 3.6 and Theorem 3.7 in \cite{Toft13}.)

\par

\begin{prop}\label{Prop:ConseqThmS}
Let $\omega ,v\in \mascP _E(\rr {2d})$ be such that $\omega$ is $v$-moderate,
$p,q\in (0,\infty ]^d$ and let $r\in (0,1]$ be such that $r\le \min (p,q)$. Then there
is an $\ep _0>0$ such that for every $\ep \in (0,\ep _0]$, there are
$\phi \in M^r_{(v)}(\rr d)$ and $\psi \in \Sigma _1(\rr d)$ such that
\begin{align}
f &= \sum _{j,\iota \in \ep \zz d} (V_\phi f)(j,\iota )
e^{i\scal {\cdo }{\iota }}\psi (\cdo -j)\notag
\\[1ex]
&=
\sum _{j,\iota \in \ep \zz d} (V_\psi f)(j,\iota )
e^{i\scal {\cdo }{\iota }}\phi (\cdo -j),
\quad
f \in M^\infty _{(\omega )}(\rr d),
\label{Eq:GabExpForm}
\end{align}
with convergence of the series with respect to the weak$^*$ topology.
Furthermore,
$$
\nm f{M^{p,q}_{(\omega )}}
\asymp
\nm {V_\phi f}{\ell ^{p,q}_{(\omega )}(\ep \zz {2d})}
\asymp
\nm {V_\psi f}{\ell ^{p,q}_{(\omega )}(\ep \zz {2d})}.
$$
If in addition $\max (p,q)<\infty$, then the series in \eqref{Eq:GabExpForm}
converge unconditionally to $f$ with respect to the $M^{p,q}_{(\omega )}$ norm.

\par

The same holds true with $W^{p,q}_{(\omega )}$ and $\ell ^{p,q}_{*,(\omega )}$
in place of $M^{p,q}_{(\omega )}$ and $\ell ^{p,q}_{(\omega )}$, respectively,
at each occurrence.
\end{prop}

\par

\subsection{Pseudo-differential operators}

\par

Next we discuss some issues in pseudo-differential calculus.
Let $\GL (d,\Omega)$ be the set of all $d\times d$-matrices with
entries in the set $\Omega$, and let $a\in \Sigma _1 
(\rr {2d})$ and $A\in \GL (d,\mathbf R)$ be fixed.
Then the \emph{pseudo-differential operator} $\op _A(a)$ with
\emph{symbol} $a$ is the linear and
continuous operator from $\Sigma _1(\rr d)$ to
$\Sigma _1(\rr d)$, defined by
\begin{equation}\label{e0.5}
(\op _A(a)f)(x)
=
(2\pi  ) ^{-d}\iint _{\rr {2d}}a(x-A(x-y),\zeta )f(y)e^{i\scal {x-y}\zeta }\, dyd\zeta ,
\end{equation}
when $f\in \Sigma _1(\rr d)$. For
general $a\in \Sigma _1'(\rr {2d})$, the
pseudo-differential operator $\op _A(a)$ is defined as the linear and
continuous operator from $\Sigma _1(\rr d)$ to $\Sigma _1'(\rr d)$ with
distribution kernel given by
\begin{equation}\label{atkernel}
K_{a,A}(x,y)=(2\pi )^{-\frac d2}(\mascF _2^{-1}a)(x-A(x-y),x-y).
\end{equation}
Here $\mascF _2F$ is the partial Fourier transform of $F(x,y)\in
\Sigma _1'(\rr {2d})$ with respect to the $y$ variable. This
definition makes sense, since the mappings
\begin{equation}\label{homeoF2tmap}
\mascF _2\quad \text{and}\quad F(x,y)\mapsto F(x-A(x-y),x-y)
\end{equation}
are homeomorphisms on $\Sigma _1'(\rr {2d})$.
In particular, the map $a\mapsto K_{a,A}$ is a homeomorphism on
$\Sigma _1'(\rr {2d})$.

\par

We set $\op _t(a)=\op _{t\cdot I}(a)$, when
$t\in \mathbf R$ and $I=I_d\in \GL (d,\mathbf R)$ is
the $d\times d$ identity matrix.
The normal or Kohn-Nirenberg representation, $a(x,D)$, is obtained
when $t=0$, and the Weyl quantization, $\op ^w(a)$, is obtained
when $t=\frac 12$. That is,
$$
a(x,D) = \op _0(a)
\quad \text{and}\quad \op ^w(a) = \op _{1/2}(a).
$$

\par

The following result explains the relationship between
$a_1$ and $a_2$ in the identity $\op _{A_1}(a_1) =
\op _{A_2}(a_2)$.
We refer to Propositions 1.1 and
 2.8 \cite{Toft15} for the proof
(see also \cite{Ho1} for background ideas).

\par

\begin{prop}\label{Prop:CalculiTransfer}
Let $p,q\in (0,\infty ]$ $A,A_1,A_2\in \GL (d,\mathbf R)$,
$\omega \in \mascP _E(\rr {4d})$ and let
\begin{equation}
\omega _A(x,\xi ,\eta ,y) = \omega (x +Ay,\xi +A^*\eta
,\eta ,y).
\end{equation}
Then the following is true:
\begin{enumerate}
\item $e^{i\scal {AD_\xi}{D_x }}$ is homeomorphic
on $\Sigma _1(\rr {2d})$ and uniquely extendable to
a homeomorphism on $\Sigma _1'(\rr {2d})$;

\vrum

\item if $a_1,a_2\in \Sigma _1'(\rr {2d})$, then
\begin{equation}\label{calculitransform}
\op _{A_1}(a_1) = \op _{A_2}(a_2) \quad \Leftrightarrow \quad
e^{i\scal {A_1D_\xi}{D_x }}a_1(x,\xi )
=e^{i\scal {A_2D_\xi}{D_x }}a_2(x,\xi );
\end{equation}

\vrum 

\item $e^{i\scal {AD_\xi}{D_x }}$ on $\Sigma _1'(\rr {2d})$
restricts to a homoemorphism from
$M^{p,q}_{(\omega )}(\rr {2d})$ to
$M^{p,q}_{(\omega _A)}(\rr {2d})$.
\end{enumerate}
\end{prop}

\par

Note here that the latter equality in \eqref{calculitransform} 
makes sense since it is equivalent to
$$
e^{i\scal {A_2x}{\xi}}\widehat a_2(\xi ,x)
=e^{i\scal {A_1x}{\xi}}\widehat a_1(\xi ,x),
$$
and that the map $a\mapsto e^{i\scal {Ax} \xi }a$ is continuous on
$\Sigma _1 '(\rr {2d})$ (cf. e.{\,}g. \cite{Tr,CaTo}). 

\par

The next result is a slight extension
of \cite[Theorem 3.1]{Toft19}, and follows
from \cite[Theorem 3.1]{Toft19} and
Proposition \ref{Prop:CalculiTransfer}. The details are
left for the reader.

\par

\begin{thm}\label{Thm:OpCont}
Let $A\in \GL (d,\mathbf R)$, $\sigma \in \operatorname{S}_{2d}$,
$\omega _1,\omega _2
\in \mathscr P_{E}(\rr {2d})$ and
$\omega _0\in \mathscr P_{E}(\rr {4d})$
be such that
\begin{equation}\label{Eq:WeightPseudoACond}
\frac {\omega _2(x,\xi  )}{\omega _1
(y,\eta )} \lesssim \omega _0( x-A(x-y),\eta -A^*(\eta -\xi ),
\xi -\eta ,y-x ).
\end{equation}
Also let $p _1,p _2\in (0,\infty]^{2d}$,
$p,q\in (0,\infty]$ be such that
\begin{equation}\label{Eq:pqconditions}
\frac 1{p _2}-\frac 1{p _1}
= \frac 1{p}+\min \left ( 0,\frac 1{q}-1\right ) ,
\quad
q \le \min (p_2)\le
\max (p_2) \le p,
\end{equation}
hold, and let
$a\in M^{p,q}_{(\omega _0)}(\rr {2d})$. Then
$\op _A(a)$ from $\mathcal S_{1/2}(\rr d)$ to $\mathcal 
S_{1/2}'(\rr d)$
extends uniquely to a continuous map from $M^{p_1}
_{\sigma ,(\omega _1)}(\rr d)$ to
$M^{p_2}_{\sigma ,(\omega _2)}(\rr d)$, and
\begin{equation}\label{Eq:PsDOEst}
\nm {\op _A(a)f}{M^{p_2}_{\sigma ,(\omega _2)}}
\lesssim
\nm a{M^{p,q}_{(\omega _0)}}
\nm f{M^{p_1}_{\sigma ,(\omega _1)}},
\quad a\in M^{p,q}_{(\omega _0)}(\rr {2d}),\ 
f\in M^{p_1}_{\sigma ,(\omega _1)}(\rr d).
\end{equation}
\end{thm}

\par

By choosing $\sigma$ as the identity map,
Theorem \ref{Thm:OpCont} gives continuity
properties for pseudo-differential operators acting
on modulation spaces of standard type, i.{\,}e.
acting between spaces of the form $M^{p,q}_{(\omega )}(\rr d)$,
where $p,q\in (0,\infty ]$ or $p,q\in (0,\infty ]^d$
(see e.{\,}g. \cite{Sjo,Tac,Toft2,Toft15} ).
If instead $\sigma$ equals $\sigma _2$ in \eqref{Eq:sigma2Def},
Theorem \ref{Thm:OpCont} gives continuity
properties for pseudo-differential operators acting
on Wiener amalgam type spaces of the form
$W^{p,q}_{(\omega )}(\rr d)$.
We remark that such continuity properties 
were also established in \cite{CorNic}.

\par

Since $W^{p,q}_{(\omega )}$ spaces are Fourier images
of $M^{p,q}_{(\omega )}$ spaces, any continuity
result valid for $M^{p,q}_{(\omega )}$ spaces can be
transformed into a continuity result for $W^{p,q}_{(\omega )}$
spaces (see e.{\,}g. \cite{BoDoOl} for ideas on such transitions).

\par

For $A=0$, i.{\,}e. the standard or Kohn-Nirenberg case,
\eqref{Eq:WeightPseudoACond} becomes
\begin{equation}\label{Eq:WeightPseudoA0Cond}
\frac {\omega _2(x,\xi  )}{\omega _1
(y,\eta )} \lesssim \omega _0( x,\eta ,
\xi -\eta ,y-x ),
\end{equation}
and Theorem \ref{Thm:OpCont} implies
\begin{alignat}{3}
\op _0(a)\, :\, M^{q',p'}_{(\omega _1)}(\rr d)
&\to & M^{p,q}_{(\omega _1)}(\rr d),
\qquad
a&\in M^{p,q}_{(\omega _0)}(\rr {2d}),& \ p,q &\in [1,\infty ],
\ q\le p.
\label{Eq:OpContSpec1}
\intertext{For slightly relaxed conditions on $p$ and
$q$ we also have}
\op _0(a)\, :\, M^{q',p'}_{(\omega _1)}(\rr d)
&\to & W^{p,q}_{(\omega _1)}(\rr d),
\qquad
a&\in W^{p,q}_{(\omega _0)}(\rr {2d}),& \ p,q &\in [1,\infty ],
\label{Eq:OpContSpec1Modif}
\end{alignat}
concerning mapping properties for
pseudo-differential operators with symbols
in Wiener amalgam type spaces $W^{p,q}_{(\omega _0)}(\rr {2d})$.
(See e.{\,}g. \cite[Theorem 3.9]{TeoTof}.)
We observe that the condition $q\le p$ in
\eqref{Eq:OpContSpec1} is removed in
\eqref{Eq:OpContSpec1Modif}.

\par

\subsection{Pseudo-differential operators of amplitude types}

\par

Let $a\in \Sigma _1(\rr {3d})$. Then the pseudo-differential
operator $\op (a)$ of \emph{amplitude type} with
\emph{amplitude} $a$ is the linear
and continuous operator from $\Sigma _1(\rr d)$ to
$\Sigma _1(\rr d)$, defined by
\begin{equation}\label{Eq:PseudoDiffAmplDef}
(\op (a)f)(x) = (2\pi )^{-d}\iint _{\rr {2d}}a(x,y,\zeta )f(y)
e^{i\scal {x-y}\zeta }\, dyd\zeta .
\end{equation}

\par

The definition of $\op (a)$ extends to more general $a$. For example,
in Section \ref{sec3} we observe that $\op (a)$ makes sense
when $a$ belongs to certain modulation spaces.

\par

Let $A\in \GL (d,\mathbf R)$ be fixed.
It is evident that the definition of $\op _A(a_0)$ is a special case
of $\op (a)$, since we may choose
$$
a(x,y,\zeta ) = a_0(x-A(x-y),\zeta ).
$$
On the other hand, it follows by Fourier inversion formula in
combination with kernel theorems for functions and distributions,
it follows that any continuous and linear operator from
$\Sigma _1(\rr d)$ to $\Sigma _1'(\rr d)$ is given by
$\op _A(a_0)$ for a unique $a_0\in \Sigma _1'(\rr {2d})$.
Consequently, the set of linear and continuous operators
from $\Sigma _1(\rr d)$ to $\Sigma _1'(\rr d)$ is not enlarged
when passing from operators of the form $\op _A(a_0)$
into the form $\op (a)$.

\par

In particular, if $a\in \Sigma _1(\rr {3d})$, then there is a unique
$a_0(\rr {2d})\in \Sigma _1'(\rr {2d})$ such that $\op (a)=\op _0(a_0)$.
By straight-forward computations it follows that
\begin{align}
\op _0(a_0)&=\op (a)
\label{Eq:aToa0A}
\intertext{when}
a_0(x,\zeta ) &= (e^{i\scal {D_\zeta}{D_y}}a)(x,y,\zeta ) \Big \vert _{y=x}
=
(e^{i\scal {D_\zeta}{D_y}}a)(x,x+y,\zeta ) \Big \vert _{y=0}
\label{Eq:aToa0B}
\end{align}
(see e.{\,}g. \cite{Ho1}).

\par

\section{Trace properties of modulation spaces}\label{sec2}

\par

In this section we deduce continuity properties of trace mappings
on modulation spaces. Especially we extend such mapping properties
to modulation spaces with general moderate weights.

\par

More precisely, for any fixed $z\in \rr{d_2}$, we consider
continuity of the trace function map which takes a suitable function
or distribution $f(x_1,x_2,x_3)$ into 
\begin{equation}\label{Eq:TraceMap}
(\Trm _{z}f)(x_1,x_3) \equiv f(x_1,z,x_3),
\qquad x_1\in \rr {d_1},\ x_3\in \rr {d_3}.
\end{equation}
Here $z\in \rr {d_2}$ is fixed and $x_j\in \rr {d_j}$ are variables, $j=1,2,3$.
By straight-forward computations it follows that $\Trm _{z}$ is
a linear and continuous map from $C^\infty (\rr {d_1+d_2+d_3})$ to 
$C^\infty (\rr {d_1+d_3})$, and that similar fact holds with $\Sigma _s$,
$\maclS _s$ or $\mascS$ in place of $C^\infty$ at each occurrence.

\par

\begin{rem}\label{Rem:TraceOpSTFT}
Let $d=d_1+d_2+d_3$, $z\in \rr {d_2}$ be fixed, $x_j\in \rr {d_j}$ for $j=1,2,3$,
$$
f\in \Sigma _1(\rr {d})
\quad \text{and}\quad
g_z(x_0)=f(x_1,z,x_3),
\quad x_0=(x_1,x_3).
$$
Also let $d_0=d_1+d_3$,
$$
\phi _0\in \Sigma _1(\rr {d_0})\setminus 0,
\quad
\phi _2\in \Sigma _1(\rr {d_2})\setminus 0
\quad \text{and}\quad
\phi (x_1,x_2,x_3)= \phi _0(x_0)\phi _2(x_2).
$$
If $\xi =(\xi _1,\xi _2,\xi _3)\in \rr d$, $\xi _j\in \rr {d_j}$
and $\xi _0=(\xi _1,\xi _3)$, then
\begin{align}
(V_{\phi _0}g_z)(x_0,\xi _0)
&=
(2\pi )^{-\frac {d_2}2}\nm {\phi _2}{L^2}^{-2}
\int _{\rr {2d_2}} V_{\phi _0} f_Y(x_0,\xi _0)
\phi _2(z-y)e^{-i\scal z\eta} \, dY ,
\label{Eq:STFTTrace}
\intertext{where}
f_Y(x_0)
&=
(V_{\phi _2}(f(x_1,\cdo ,x_3)))(y,\eta ),\qquad Y=(y,\eta )\in \rr {2d_2}.
\label{Eq:STFTTrace2}
\end{align}
In fact, \eqref{Eq:STFTTrace} follows by first evaluating the integral with
respect to $\eta$ and using Parseval's formula, and thereafter integrate with
respect to $y$.
For future references we notice that \eqref{Eq:STFTTrace} is the same as
\begin{multline}\tag*{(\ref{Eq:STFTTrace})$'$}
(V_{\phi _0}g_z)(x_0,\xi _0)
\\[1ex]
=
(2\pi )^{-\frac {d_2}2}\nm {\phi _2}{L^2}^{-2}
\iint _{\rr {2d_2}} V_{\phi} f(x_1,y,x_3,\xi _1,\eta ,\xi _3)
\phi _2(z-y)e^{-i\scal z\eta} \, dyd\eta .
\end{multline}
\end{rem}

\par

For modulation spaces we have the following trace result.
Here the involved Lebesgue exponents and weight functions should satisfy
conditions of the form
\begin{alignat}{1}
p_0 = (p_1,p_3), \quad p &= (p_1,p_2,p_3),  \quad q_0 = (q_1,q_3),  \quad
q = (q_1,q_2,q_3),
\label{Eq:CondTraceLeb}
\\[1ex]
\max \left ( \frac 1{p_0},\frac 1{q_1},\frac 1{q_2},1\right )
-\frac 1{q_2}
&\le \frac 1r,
\label{Eq:CondTraceLebr1}
\end{alignat}
\begin{equation}\label{Eq:CondTraceWeight}
\sup _{\xi _3\in \rr {d_3}}
\left (
\NM {\sup _{(x,\xi _1)\in \rr {d+d_1}}
\left (
\frac {\omega _0(x_0,\xi _0)e^{-r_0|x_2|}}
{\omega (x,\xi _1,\cdo ,\xi _3)}
\right )
}{L^{r}(\rr {d_2})}
\right )
< \infty
\end{equation}
and
\begin{equation}\label{Eq:CondTraceWeight1A}
\omega (x,\xi )
\lesssim
\omega _0(x_0,\xi _0)e^{r_0(|x_2|+|\xi _2|)}.
\end{equation}
Here
\begin{equation}\label{Eq:CoordAbbreviation}
\begin{alignedat}{2}
x &=(x_1,x_2,x_3)\in \rr {d_1}\times \rr {d_2}\times \rr {d_3}, &
\quad
x_0 &=(x_1,x_3)
\\[1ex]
\xi &= (\xi _1,\xi _2,\xi _3)\in \rr {d_1}\times
\rr {d_2}\times \rr {d_3} &
\quad \text{and}\quad
\xi _0 &=(\xi _1,\xi _3).
\end{alignedat}
\end{equation}
We observe that \eqref{Eq:CondTraceWeight} is the same
as
\begin{align}
\omega _0(x_0,\xi _0)e^{-r_0|x_2|}
&\lesssim 
\omega (x,\xi )\vartheta (\xi _2,\xi _3),
\qquad
C_\vartheta \equiv
\sup _{\xi _3\in \rr {d_3}}\nm {\vartheta (\cdo ,\xi _3)}{L^{r}}<\infty ,
\label{Eq:CondTraceWeightModif}
\intertext{where $\vartheta \in \mascP _E(\rr {d_2+d_3})$
is given by}
\vartheta (\xi _2,\xi _3)
&\equiv
\sup _{x\in \rr d}\left (\sup _{\xi _1\in \rr {d_1}}
\left (
\frac {\omega _0(x_0,\xi _0)e^{-r_0|x_2|}}
{\omega (x,\xi )}
\right )
\right ),
\label{Eq:varthetaDef}
\end{align}
which indicates similarities between the conditions
\eqref{Eq:CondTraceWeight} and \eqref{Eq:CondTraceWeight1A}. We observe

\par

\begin{thm}\label{Thm:TraceMod}
Let $d_j\ge 0$ be integers, $z\in \rr {d_2}$ be fixed,
$p_j,q_j\in (0,\infty ]^{d_j}$, $r\in (0,\infty ]^{d_2}$, 
$j=1,2,3$,
be such that \eqref{Eq:CondTraceLebr1} holds
and let
$\omega \in \mascP _E(\rr {2(d_1+d_2+d_3)})$,
$\omega _0 \in \mascP _E(\rr {2(d_1+d_3)})$,
$p$, $p_0$, $q$ and $q_0$
be such that \eqref{Eq:CondTraceLeb} and \eqref{Eq:CondTraceWeight}
hold true for some $r_0\ge 0$.
Then the following is true:
\begin{enumerate}
\item
the map $\Trm _{z}$ from
$\Sigma _1  (\rr {d_1+d_2+d_3})$ to
$\Sigma _1(\rr {d_1+d_3})$ extends uniquely to a continuous map from
$M^{p,q}_{(\omega )}(\rr {d_1+d_2+d_3})$ to
$M^{p_0,q_0}_{(\omega _0)}(\rr {d_1+d_3})$, and
\begin{alignat}{2}
\nm {\Trm _{z}f}{M^{p_0,q_0}_{(\omega _0)}}
& \lesssim &
\nm f{M^{p,q}_{(\omega )}},
\qquad
f &\in M^{p,q}_{(\omega )}(\rr {d_1+d_2+d_3})
\label{Eq:TraceModEst}
\end{alignat}

\vrum

\item if in addition \eqref{Eq:CondTraceWeight1A} holds true
for some $r_0\ge 0$,
then the map in {\rm{(1)}} a surjective.
\end{enumerate}
\end{thm}

\par

For modulation spaces of Wiener-amalgam types
we also have the following. Here the conditions
\eqref{Eq:CondTraceLebr1} and 
\eqref{Eq:CondTraceWeight} are relaxed into
\begin{equation}
\max \left ( \frac 1{q_1},\frac 1{q_2},1\right )
-\frac 1{q_2}
\le \frac 1r,
\label{Eq:CondTraceLebr2}
\end{equation}
and
\begin{equation}\tag*{(\ref{Eq:CondTraceWeight})$'$}
\sup _{(x,\xi _3)\in \rr {d+d_3}}
\left (
\NM {\sup _{\xi _1\in \rr {d_1}}
\left (
\frac {\omega _0(x_1,x_3,\xi _1,\xi _3)e^{-r_0|x_2|}}
{\omega (x_1,x_2,x_3,\xi _1,\cdo ,\xi _3)}
\right )
}{L^{r}(\rr {d_2})}
\right )
< \infty .
\end{equation}

\par

\begin{thm}\label{Thm:TraceMod2}
Let $d_j\ge 0$ be integers, $z\in \rr {d_2}$ be fixed,
$p_j,q_j\in (0,\infty ]^{d_j}$, $r\in (0,\infty ]^{d_2}$, 
$j=1,2,3$,
be such that \eqref{Eq:CondTraceLebr2} holds
and let
$\omega \in \mascP _E(\rr {2(d_1+d_2+d_3)})$,
$\omega _0 \in \mascP _E(\rr {2(d_1+d_3)})$,
$p$, $p_0$, $q$ and $q_0$
be such that \eqref{Eq:CondTraceLeb} and \eqref{Eq:CondTraceWeight}$'$
hold true for some $r_0\ge 0$.
Then the following is true:
\begin{enumerate}
\item
the map $\Trm _{z}$ from
$\Sigma _1  (\rr {d_1+d_2+d_3})$ to
$\Sigma _1(\rr {d_1+d_3})$ extends uniquely to a continuous map from
$W^{p,q}_{(\omega )}(\rr {d_1+d_2+d_3})$ to
$W^{p_0,q_0}_{(\omega _0)}(\rr {d_1+d_3})$, and
\begin{alignat}{2}
\nm {\Trm _{z}f}{W^{p_0,q_0}_{(\omega _0)}}
& \lesssim &
\nm f{W^{p,q}_{(\omega )}},
\qquad
f &\in W^{p,q}_{(\omega )}(\rr {d_1+d_2+d_3}) \text ;
\end{alignat}

\vrum

\item if in addition \eqref{Eq:CondTraceWeight1A} holds true
for some $r_0\ge 0$,
then the map in {\rm{(1)}} is surjective.
\end{enumerate}
\end{thm}

\par

For the proof of Theorems \ref{Thm:TraceMod} and \ref{Thm:TraceMod2}
we recall the Young type inequality
\begin{equation}\label{Eq:YoungSpec}
\nm {f_1*f_2}{\ell ^p_{(\omega )}}
\le
\nm {f_1}{\ell ^p_{(\omega )}}\nm {f_2}{\ell ^{\min (p,1)}_{(v)}},
\qquad
f_1\in \ell ^p_{(\omega )}(\ep \zz d),\ f_2\in \ell ^{\min (p,1)}_{(\omega )}(\ep \zz d)
\end{equation}
for discrete Lebesgue spaces, when $\omega ,v\in \mascP _E(\rr d)$ satisfy
$\omega (x+y)\le \omega (x)v(y)$. This gives
\begin{equation}\label{Eq:ConvExp}
\nm {f*e^{-r|\cdo |}}{\ell ^p_{(\omega )}(\ep \zz d)}
\lesssim
\nm {f}{\ell ^p_{(\omega )}(\ep \zz d)},
\qquad
f\in \ell ^p_{(\omega )}(\ep \zz d),
\end{equation}
provided $r>0$ is chosen large enough.

\par

\begin{rem}\label{Rem:pNorms}
Let $p\in (0,1]$, $\mascB$ be a vector space and let
$\nm \cdo \mascB$ be a $p$-norm, i.{\,}e. $\nm \cdo \mascB$
is a quasi-norm on $\mascB$ which fulfills
\begin{alignat*}{2}
\nm {f+g}{\mascB}^p
&\le
\nm {f}{\mascB}^p + \nm {g}{\mascB}^p,
&\qquad f,g &\in \mascB .
\intertext{Then it follows that}
\nm {f+g}{\mascB}^q
&\le
\nm {f}{\mascB}^q + \nm {g}{\mascB}^q,
&\qquad f,g &\in \mascB ,
\end{alignat*}
for every $q\in (0,p]$.
\end{rem}

\par

\begin{proof}[Proof of Theorem \ref{Thm:TraceMod}]
First we prove (1).
Let $d=d_1+d_2+d_3$, $d_0=d_1+d_3$, $\rho = \min (p_0,q_1,q_2,1)$,
$\rho _0=\min (p,q,1)=\min (\rho ,q_3)$
and $g_z=\Trm _{z}f$ and $v\in \mascP _E(\rr {2d})$ be such that
$\omega$ is $v$-moderate. Also let $x$, $\xi$, $x_0$, $\xi _0$ and
$\vartheta$ be given by \eqref{Eq:CoordAbbreviation} and
\eqref{Eq:varthetaDef}.
Then $\vartheta \in \mascP _E (\rr {d_2+d_3})$ fulfills
\eqref{Eq:CondTraceWeightModif}.

\par

First suppose that $f\in \Sigma _1(\rr d)$.
By Proposition \ref{Prop:ConseqThmS}, there are
$\phi _m\in M^{\rho _0}_{(v)}(\rr d)$, $\psi _m\in \Sigma _1(\rr {d_m})$,
$m=1,2,3$ such that
\begin{align}
f(x)
&=
c_\ep \sum _{j,\iota \in \ep \zz {d}}
V_\phi f(j,\iota )e^{i\scal {x}{\iota}}\psi (x-j),
\notag
\intertext{where}
\phi &= \phi _1\otimes \phi _2\otimes \phi _3
\qquad \text{and}\qquad
\psi =\psi _1\otimes \psi _2\otimes \psi _3.
\notag
\intertext{We have}
g_z(x_1,x_3)
&=
c_\ep \sum _{j,\iota \in \ep \zz {d}}
V_\phi f(j,\iota )e^{i(\scal {x_1}{\iota _1}+\scal {z}{\iota _2}+\scal {x_3}{\iota _3})}
\psi (x_1-j_1,z-j_2,x_3-j_3),
\label{Eq:TraceDist}
\intertext{when}
g_z(x_1,x_3) &= f(x_1,z,x_3),
\end{align}
and observe that the series possess strong convergence properties,
because
$$
|V_\phi f(j,\iota )| \lesssim e^{-r_1(|j|+|\iota |)}
\quad \text{and}\quad
|\psi (x_1-j_1,z-j_2,x_3-j_3)|
\lesssim
e^{-r_1(|x_1-j_1|+|j_2|+|x_3-j_3|)}
$$
for every $r_1>0$.

\par

An application of the short-time Fourier transform
on \eqref{Eq:TraceDist} gives
\begin{multline}\label{Eq:STFTTraceMap}
V_{\psi _0}g_z(x_0,\xi _0) 
\\[1ex]
=
c_\ep \sum _{j,\iota \in \ep \zz {d}}
V_\phi f(j,\iota )e^{i(\scal {z}{\iota _2} +\scal {j_0}{\iota _0-\xi _0})}
(V_{\psi _0}\psi _0)(x_0-j_0,\xi _0-\iota _0)\psi _2(z-j_2),
\end{multline}
where
$$
j_0=(j_1,j_3),\quad \iota _0=(\iota _1,\iota _3)
\quad \text{and}\quad
\psi _0=\psi _1\otimes \psi _3.
$$
By letting
$$
F _\omega = |V_\phi f \cdot \omega |
\quad \text{and}\quad
F _{0,\omega _0} = |V_{\psi _0} g_z \cdot \omega _0|,
$$
and using that
\begin{equation*}
\omega _0(x_0,\xi _0) \lesssim \omega (x_1,j_2,x_3,\xi _1,\iota _2,\xi _3)
e^{r_0|j_2|}\vartheta (\iota _2,\lambda _3)
\\[1ex]
\lesssim
\omega (j,\iota )e^{r_0(|x_0-j_0|+|j_2|+|\xi _0-\iota _0|)}\vartheta (\iota _2,\lambda _3).
\end{equation*}
and
$$
|(V_{\psi _0}\psi _0)(x_0-j_0,\xi _0-\iota _0)\psi _2(z-j_2)
e^{r_0(|x_0-j_0|+|j_2|+|\xi _0-\iota _0|)}|
\lesssim
e^{-r_1(|x_0-j_0|+|j_2|+|\xi _0-\iota _0|)}
$$
for every $r_1>0$, it follows from \eqref{Eq:STFTTraceMap} that
\begin{equation}\label{Eq:TraceMapBasicModEst}
F _{0,\omega _0}(l_0,\lambda _0)
\lesssim
\sum _{j,\iota \in \ep \zz {d}} 
F _\omega (j,\iota ) e^{-r_1(|l_0-j_0|+|j_2|+|\lambda _0-\iota _0|)}\vartheta (\iota _2,\lambda _3),
\qquad r_1>0,
\end{equation}
where
$$
l_0=(l_1,l_3)\in \ep \zz {d_0}
\quad \text{and} \quad
\lambda _0=(\lambda _1,\lambda _3)\in \ep \zz {d_0}.
$$
Since $\rho \in (0,1]$, \eqref{Eq:TraceMapBasicModEst} gives
\begin{equation}\tag*{(\ref{Eq:TraceMapBasicModEst})$'$}
F _{0,\omega _0}(l_0,\lambda _0)^\rho
\lesssim
\sum _{j,\iota \in \ep \zz {d}} 
F _\omega (j,\iota )^\rho e^{-r_1(|l_0-j_0|+|j_2|+|\lambda _0-\iota _0|)}
\vartheta (\iota _2,\lambda _3)^\rho ,
\qquad r_1>0.
\end{equation}

\par

Now let
\begin{gather*}
\Lambda _1 = \ep \zz {d_2+d_3+d},\quad
\Lambda _2 = \ep \zz {d_3+d}, \quad
\Lambda _3 = \ep \zz {d_2+d_3},
\\[1ex]
F_{p_1,\omega}(x_2,x_3,\xi )
=
\nm {F_\omega (\cdo ,x_2,x_3,\xi )}{\ell ^{p_1}(\ep \zz {d_1})},
\quad
F_{p_1,p_2,\omega}(x_3,\xi )
=
\nm {F_\omega (\cdo ,x_3,\xi )}{\ell ^{p_1,p_2}(\ep \zz {d_1+d_2})},
\\[1ex]
F_{p,\omega}(\xi )
=
\nm {F_\omega (\cdo ,\xi )}{\ell ^{p}(\ep \zz {d})},
\quad
F_{p,q_1,\omega}(\xi _2,\xi _3)
=
\nm {F_\omega (\cdo ,\xi _2,\xi _3)}{\ell ^{p,q_1}(\ep \zz {d+d_1})},
\\[1ex]
F_{p,q_1q_2,\omega}(\xi _3)
=
\nm {F_\omega (\cdo ,\xi _3)}{\ell ^{p,q_1,q_2}(\ep \zz {d+d_1+d_2})}
\intertext{and}
G_{r_1,\rho ,\omega}(l_1,j_2,j_3,\iota )
=
(F_\omega (\cdo ,j_2,j_3,\iota )^\rho * e^{-r_1|\cdo |})(l_1),
\end{gather*}
where the convolution is the discrete convolution with
respect to $\ep \zz {d_1}$.
By \eqref{Eq:TraceMapBasicModEst}$'$ it follows that
$$
F _{0,\omega _0}(l_0,\lambda _0)^\rho
\lesssim
\sum _{(j_2,j_3,\iota )\in \Lambda _1} 
(F _\omega (\cdo ,j_2,j_3,\iota )^\rho *
e^{-r_1|\cdo |})(l_1)e^{-r_1(|j_2|+|l_3-j_3|+|\lambda _0-\iota 
_0|)}\vartheta (\iota _2,\lambda _3)^\rho ,
$$
for every $r_1>0$.
If we apply the $\ell ^{p_1/\rho}$ norm with respect to $l_1$
variable, and use
Minkowski's, Young's and H{\"o}lder's inequalities, we obtain
\begin{multline*}
\nm {F _{0,\omega _0}(\cdo , l_3,\lambda _0)^\rho}{\ell ^{p_1/\rho}}
=
\nm {F _{0,\omega _0}(\cdo , l_3,\lambda _0)}{\ell ^{p_1}}^\rho
\\[1ex]
\lesssim
\NM {\sum _{(j_2,j_3,\iota )\in \Lambda _1} 
G_{r_1,\rho ,\omega}(\cdo ,j_2,j_3,\iota )
e^{-r_1(|j_2|+|l_3-j_3|+|\lambda _0-\iota _0|)}
\vartheta (\iota _2,\lambda _3)^\rho}{\ell ^{p_1/\rho}}
\\[1ex]
\lesssim
\sum _{(j_2,j_3,\iota )\in \Lambda _1}
\left (
\nm {F _\omega (\cdo ,j_2,j_3,\iota )^\rho *
e^{-r_1|\cdo |}}{\ell ^{p_1/\rho}}
\right )
e^{-r_1(|j_2|+|l_3-j_3|+|\lambda _0-\iota _0|)}
\vartheta (\iota _2,\lambda _3)^\rho
\\[1ex]
\lesssim
\sum _{(j_2,j_3,\iota )\in \Lambda _1}
\left (
\nm {F _\omega (\cdo ,j_2,j_3,\iota )^\rho}{\ell ^{p_1/\rho}}
\right )
e^{-r_1(|j_2|+|l_3-j_3|+|\lambda _0-\iota _0|)}
\vartheta (\iota _2,\lambda _3)^\rho
\\[1ex]
=
\sum _{(j_2,j_3,\iota )\in \Lambda _1}
F _{p_1,\omega} (j_2,j_3,\iota )^\rho
e^{-r_1(|j_2|+|l_3-j_3|+|\lambda _0-\iota _0|)}
\vartheta (\iota _2,\lambda _3)^\rho
\\[1ex]
\le
\sum _{(j_3,\iota )\in \Lambda _2}
\nm {F _{p_1,\omega} (\cdo ,j_3,\iota )^\rho}
{\ell ^{\infty}}
\nm {e^{-r_1|\cdo |}}{\ell ^1}
e^{-r_1(|l_3-j_3|+|\lambda _0-\iota _0|)}
\vartheta (\iota _2,\lambda _3)^\rho
\\[1ex]
\lesssim
\sum _{(j_3,\iota )\in \Lambda _2}
F _{p_1,p_2,\omega} (j_3,\iota )^\rho
e^{-r_1(|l_3-j_3|+|\lambda _0-\iota _0|)}
\vartheta (\iota _2,\lambda _3)^\rho ,
\end{multline*}
for every $r_1>0$. In the last inequality
we have used the fact that $\nm f{\ell ^\infty
(\ep \zz {d_2})} \le \nm f{\ell ^p(\ep \zz {d_2})}$
for every sequence $f$ on $\ep \zz {d_2}$
and $p\in (0,\infty ]$. Hence, if
$$
G_{p_1,p_2,r_1,\rho ,\omega}(l_3,\iota )
=
(F_{p_1,p_2,\omega} (\cdo ,\iota )^\rho * e^{-r_1|\cdo |})(l_3),
$$
where the convolution is the discrete convolution with
respect to $\ep \zz {d_3}$, we get
$$
\nm {F _{0,\omega _0}(\cdo , l_3,\lambda _0)}{\ell ^{p_1}}^\rho
\lesssim
\sum _{\iota \in \ep \zz d}
G_{p_1,p_2,r_1,\rho ,\omega}(l_3,\iota )
e^{-r_1|\lambda _0-\iota _0|}\vartheta (\iota _2,\lambda _3)^\rho ,
\quad
r_1>0.
$$

\par

An application of the $\ell ^{p_3/\rho}$ norm with respect to $l_3$ variable,
and using Minkowski's, Young's and H{\"o}lder's inequalities now give
\begin{multline*}
\nm {F _{0,\omega _0}(\cdo ,\lambda _0)}{\ell ^{p_0}}^\rho
\lesssim
\NM
{\sum _{\iota \in \ep \zz d}
G_{p_1,p_2,r_1,\rho ,\omega}(\cdo ,\iota )
e^{-r_1|\lambda _0-\iota _0|}\vartheta (\iota _2,\lambda _3)^\rho}{\ell ^{p_3/\rho}}
\\[1ex]
\lesssim
\sum _{\iota \in \ep \zz d}
\left (
\nm {F_{p_1,p_2,\omega} (\cdo ,\iota )^\rho * e^{-r_1|\cdo |}}{\ell ^{p_3/\rho}}
\right )
e^{-r_1|\lambda _0-\iota _0|}\vartheta (\iota _2,\lambda _3)^\rho
\\[1ex]
\lesssim
\sum _{\iota \in \ep \zz d}
\nm {F_{p_1,p_2,\omega} (\cdo ,\iota )^\rho}{\ell ^{p_3/\rho}}
e^{-r_1|\lambda _0-\iota _0|}\vartheta (\iota _2,\lambda _3)^\rho
\\[1ex]
=
\sum _{\iota \in \ep \zz d}
F_{p,\omega} (\iota )^\rho
e^{-r_1|\lambda _0-\iota _0|}\vartheta (\iota _2,\lambda _3)^\rho
\end{multline*}
for every $r_1>0$. That is
$$
\nm {F _{0,\omega _0}(\cdo ,\lambda _0)}{\ell ^{p_0}}^\rho
\lesssim
\sum _{(\iota _2,\iota _3)\in \Lambda _3}
G_{p,r_1,\rho ,\omega}(\lambda _1,\iota _2,\iota _3)
e^{-r_1|\lambda _3-\iota _3|}\vartheta (\iota _2,\lambda _3)^\rho ,
\quad
r_1>0,
$$
where
$$
G_{p,r_1,\rho ,\omega}(\iota )
=
(F_{p,\omega} (\cdo ,\iota _2,\iota _3)^\rho * e^{-r_1|\cdo |})(\iota _1)
$$
and the convolution is the discrete convolution with respect to $\ep \zz {d_1}$.
If we apply the $\ell ^{q_1/\rho}$ norm with respect
to $\lambda _1$ variable, and use Minkowski's, Young's and H{\"o}lder's
inequalities, we obtain
\begin{multline*}
\nm {F _{0,\omega _0}(\cdo ,\lambda _3)}{\ell ^{p_0,q_1}}^\rho
\lesssim
\NM {
\sum _{(\iota _2,\iota _3)\in \Lambda _3}
G_{p,r_1,\rho ,\omega}(\lambda _1,\iota _2,\iota _3)
e^{-r_1|\lambda _3-\iota _3|}\vartheta (\iota _2,\lambda _3)^\rho
}{\ell ^{q_1/\rho}}
\\[1ex]
\le
\sum _{(\iota _2,\iota _3)\in \Lambda _3}
\nm {F_{p,\omega} (\cdo ,\iota _2,\iota _3)^\rho * e^{-r_1|\cdo |}}{\ell ^{q_1/\rho}}
e^{-r_1|\lambda _3-\iota _3|}\vartheta (\iota _2,\lambda _3)^\rho
\\[1ex]
\lesssim
\sum _{(\iota _2,\iota _3)\in \Lambda _3}
\nm {F_{p,\omega} (\cdo ,\iota _2,\iota _3)^\rho}{\ell ^{q_1/\rho}}
e^{-r_1|\lambda _3-\iota _3|}\vartheta (\iota _2,\lambda _3)^\rho
\\[1ex]
=
\sum _{(\iota _2,\iota _3)\in \Lambda _3}
F_{p,q_1,\omega}(\iota _2,\iota _3)^\rho e^{-r_1|\lambda _3-\iota _3|}
\vartheta (\iota _2,\lambda _3)^\rho
\\[1ex]
\le
\sum _{\iota _3\in \ep \zz {d_3}}
\nm {F_{p,q_1,\omega}(\cdo ,\iota _3)^\rho}{\ell ^{q_2/\rho}}
\nm {\vartheta (\cdo ,\lambda _3)^\rho}{\ell ^{(q_2/\rho)'}}
e^{-r_1|\lambda _3-\iota _3|}
\\[1ex]
=
\sum _{\iota _3\in \ep \zz {d_3}}
\left (F_{p,q_1,q_2,\omega}(\iota _3)
\nm {\vartheta (\cdo ,\lambda _3)}{\ell ^{r}}\right )^\rho
e^{-r_1|\lambda _3-\iota _3|}
\\[1ex]
\le
C_\vartheta \sum _{\iota _3\in \ep \zz {d_3}}
F_{p,q_1,q_2,\omega}(\iota _3)^\rho
e^{-r_1|\lambda _3-\iota _3|}
\asymp
(F_{p,q_1,q_2,\omega}^\rho * e^{-r_1|\cdo |})(\lambda _3)
\end{multline*}
for every $r_1>0$. Here $C_\vartheta$ is given by
\eqref{Eq:CondTraceWeightModif}.

\par

By applying the $\ell ^{q_3/\rho}$ quasi-norm on
the latter inequality we obtain
\begin{multline*}
\nm {F _{0,\omega _0}}{\ell ^{p_0,q_0}(\ep \zz {2d_0})}^\rho
\lesssim
\NM {
F_{p,q_1,q_2,\omega}^\rho * e^{-r_1|\cdo |}
}{\ell ^{q_3/\rho}(\ep \zz {d_3})}
\\[1ex]
\le
\nm {F_{p,q_1,q_2,\omega}^\rho}{\ell ^{q_3/\rho}(\ep \zz {d_3})}
\nm {e^{-r_1|\cdo |}}{\ell ^{\min (1,q_3/\rho )}(\ep \zz {d_3})}
\end{multline*}
for every $r_1>0$, which gives
$$
\nm {F _{0,\omega _0}}{\ell ^{p_0,q_0}(\ep \zz {2d_0})}
\lesssim
\nm {F_{p,q_1,q_2,\omega}}{\ell ^{q_3}(\ep \zz {d_3})}
=
\nm {F_\omega}{\ell ^{p,q}(\ep \zz {2d})}.
$$
This gives \eqref{Eq:TraceModEst} for $f\in \Sigma (\rr {d_1+d_2+d_3})$,
i.{\,}e.
\begin{equation}\tag*{(\ref{Eq:TraceModEst})$'$}
\nm {\Trm _{z}f}{M^{p_0,q_0}_{(\omega _0)}}
\lesssim
\nm f{M^{p,q}_{(\omega )}},
\qquad
f \in \Sigma _1(\rr {d_1+d_2+d_3})
\end{equation}
The assertion now follows in the case $\max (p,q)<\infty$
from \eqref{Eq:TraceModEst}$'$ and the fact that
$\Sigma _1(\rr {d_1+d_2+d_3})$
is dense in $M^{p,q}_{(\omega )}(\rr {d_1+d_2+d_3})$
in view of (1) in Proposition \ref{Prop:ModSpacesBasicProp}.

\par

For the case $\max (p,q)=\infty$, we first assume that $\min (p,q)\ge 1$.
Then all involved modulation spaces are Banach spaces, and it follows
by \eqref{Eq:TraceModEst}$'$ and Hahn-Banach's theorem that $\Trm _{z}$
extends to a continuous map from
$M^{p,q}_{(\omega )}(\rr {d_1+d_2+d_3})$ to
$M^{p_0,q_0}_{(\omega _0)}(\rr {d_1+d_3})$.

\par

In order to show that this extension is unique, let
\begin{align*}
\omega _{0,1}(x_0,\xi _0)
&=
\omega _0(x_0,\xi _0)e^{-r_0(|x_0|+|\xi _0|)}
\intertext{and}
\omega _1(x,\xi )
&=
\omega (x,\xi )e^{-r_0(|x|+|\xi _0|)}
\vartheta (\xi _2,\xi _3).
\end{align*}
Then
\begin{equation}\label{Eq:ModEmbFeichAlg}
M^{p,q}_{(\omega )}(\rr d) \hookrightarrow M^{1,1}_{(\omega _1)}(\rr d),
\qquad M^{p_0,q_0}_{(\omega _0)}(\rr {d_0})
\hookrightarrow M^{1,1}_{(\omega _{0,1})}(\rr {d_0}),
\end{equation}
and \eqref{Eq:CondTraceWeight} implies
$$
\sup _{\xi _3}
\left (
\NM {\sup _{x,\xi _1} 
\left (
\frac {\omega _{0,1}(x_0,\xi _0)e^{-r_0|x_2|}}
{\omega (x,\xi _1,\cdo ,\xi _3)}
\right )
}{L^{\infty}(\rr {d_2})}
\right )
< \infty
$$
(see also \eqref{Eq:commdiagram}).
Since the assertion holds in the case $\max (p,q)<\infty$, it now follows
that $\Trm _{z}$ is uniquely defined and continuous
from $M^{1,1}_{(\omega _1)}(\rr d)$ to $M^{1,1}_{(\omega _{0,1})}(\rr {d_0})$.
In particular, if $f\in M^{p,q}_{(\omega )}(\rr d)$, then $\Trm _{z}f$ is uniquely
defined as an element in $M^{1,1}_{(\omega _{0,1})}(\rr {d_0})$ due to the
inclusions above. The asserted uniqueness now follows from the fact that
if $g_1,g_2\in M^{p_0,q_0}_{(\omega _0)}(\rr {d_0})$ and
$g_1=g_2$ as elements in $M^{1,1}_{(\omega _{0,1})}(\rr {d_0})$,
then $g_1=g_2$ as elements in $M^{p_0,q_0}_{(\omega _0)}(\rr {d_0})$.

\par

Finally suppose that $p$ and $q$ are general, and let
$\omega _1$ and
$\omega _{0,1}$ be as above. Since
\eqref{Eq:ModEmbFeichAlg} holds
when $\min (p,q)\ge 1$ and that $M^{p,q}_{(\omega)}(\rr d)$
increases with $p$ and $q$, it follows that \eqref{Eq:ModEmbFeichAlg}
still holds without the restriction $\min (p,q)\ge 1$. Hence,
if $f\in M^{p,q}_{(\omega)}(\rr d)$, then $\Trm _{z}f$ is uniquely defined
as an element in $M^{1,1}_{(\omega _{0,1})}(\rr {d_0})$. By the
computations which lead to \eqref{Eq:TraceModEst}$'$, it follows that
$\Trm _{z}f\in M^{p_0,q_0}_{(\omega _0)}(\rr {d_0})$, and that
\eqref{Eq:TraceModEst} holds. This gives (1).

\par

It remains to prove (2). Let $f_0\in M^{p,q}_{(\omega _0)}(\rr {d_0})$
be fixed, $\phi _j$ and $\phi$ be as above,
$\fy \in \Sigma _1(\rr {d_2})$ be such that $\fy (0)=1$,
and let
$$
f(x_1,x_2,x_3 ) = f_0(x_1,x_3)\fy (x_2-z).
$$
We have
$$
V_{\phi _2}\fy (x_2,\xi _2 )\lesssim e^{-2r_1(|x_2|+|\xi _2 |)}
$$
for every $r_1>r_0$, in view of Lemma \ref{Lemma:GSFourierest}.
Hence, by letting
$$
F = |V_\phi f |\cdot \omega ,
\qquad
F_0 = |V_{\phi _0} f_0 |\cdot \omega _0
$$
and using \eqref{Eq:CondTraceWeight1A}, we obtain
\begin{multline*}
F(x_1,x_2,x_3,\xi _1,\xi _2,\xi _3)
=
|V_{\phi _0}f_0(x_1,x_3,\xi _1,\xi _3) V_{\phi _2}
\fy (x_2-z,\xi _2 )\omega (x_1,x_2,x_3,\xi _1,\xi _2,\xi _3)|
\\[1ex]
\lesssim
F_0(x_1,x_3,\xi _1,\xi _3) e^{r_0|x_2|}e^{-2r(|x_2-z| +|\xi _2 |)}
\lesssim
F_0(x_1,x_3,\xi _1,\xi _3)e^{-r_1(|x_2| +|\xi _2|)}.
\end{multline*}

\par

An application of the $L^{p,q}$ norm on the latter inequality
now gives 
$$
\nm f{M^{p,q}_{(\omega )}}
\lesssim
\nm g{L^{p_2,q_2}}\nm {f_0}{M^{p_0,q_0}_{(\omega _0)}}
\asymp
\nm {f_0}{M^{p_0,q_0}_{(\omega _0)}},
\quad g(x_2,\xi _2) = e^{-r_1(|x_2| +|\xi _2|)},
$$
which in turn implies that $f\in M^{p,q}_{(\omega )}(\rr {3d})$,
and the surjectivity follows.
\end{proof}

\par

\begin{proof}[Proof of Theorem \ref{Thm:TraceMod2}]
The result follows by similar arguments as
in the proof of Theorem \ref{Thm:TraceMod}.
In order for clarifying some details, we
present the first part of the proof.

\par

We use the same notation as in the proof of Theorem 
\ref{Thm:TraceMod}, except that we let
$\rho =\min (q_1,q_2,1)$,
$$
\vartheta (x,\xi _2,\xi _3)
\equiv
\sup _{\xi _1\in \rr {d_1}}
\left (
\frac {\omega _0(x_0,\xi _0)e^{-r_0|x_2|}}
{\omega (x,\xi )}
\right ) \in \mascP _E(\rr {d+d_2+d_3}),
$$
\begin{gather*}
\Lambda _1 = \ep \zz {d+d_2+d_3},\quad
\Lambda _2 = \ep \zz {d+d_3}, \quad
\Lambda _3 = \ep \zz {d_2+d_3},
\\[1ex]
F_{q_1,\omega}(x,\xi _2,\xi _3)
=
\nm {F_\omega (x,\cdo ,\xi _2,\xi _3)}
{\ell ^{q_1}(\ep \zz {d_1})},
\intertext{and}
G_{r_1,\rho ,\omega}(j,\lambda _1,\iota _2,\iota _3)
=
(F_\omega (j,\cdo ,\iota _2,\iota _3)^\rho *
e^{-r_1|\cdo |})(\lambda _1),
\end{gather*}
where the convolution is the discrete convolution with
respect to $\ep \zz {d_1}$. Then
\begin{equation}\label{Eq:OtherWeightCond1}
\omega _0(x_0,\xi _0)e^{-r_0|x_2|}
\le
\vartheta (x,\xi _2,\xi _3)\omega (x,\xi ),
\end{equation}
and \eqref{Eq:CondTraceWeight}$'$ implies
\begin{equation}\tag*{(\ref{Eq:CondTraceWeight})$''$}
C_\vartheta = \sup _{x,\xi _3}
\nm {\vartheta (x,\cdo ,\xi _3)}{L^{r}}
<\infty .
\end{equation}

\par

Suppose that $f\in \Sigma _1(\rr d)$.
By \eqref{Eq:TraceMapBasicModEst}$'$
and \eqref{Eq:OtherWeightCond1} it follows that
\begin{multline*}
F _{0,\omega _0}(l_0,\lambda _0)^\rho
\\[1ex]
\lesssim
\sum _{(j,\iota _2,\iota _3)\in \Lambda _1} 
(F _\omega (j,\cdo \iota _2,\iota _3)^\rho *
e^{-r_1|\cdo |})
(\lambda _1)e^{-r_1(|l_0-j_0|+|j_2|
+|\lambda _3-\iota _3|)}
\vartheta (j,\iota _2,\lambda _3)^\rho ,
\quad r_1>0.
\end{multline*}
If we apply the $\ell ^{q_1/\rho}$ norm with respect to 
$\lambda _1$ variable, and use
Minkowski's and Young's inequalities,
we obtain
\begin{multline*}
\nm {F _{0,\omega _0}(l_0,\cdo ,\lambda _3)
^\rho}{\ell ^{q_1/\rho}}
=
\nm {F _{0,\omega _0}
(l_0,\cdo ,\lambda _3)}{\ell ^{q_1}}^\rho
\\[1ex]
\lesssim
\NM {\sum _{(j,\iota _2,\iota _3)\in \Lambda _1} 
G_{r_1,\rho ,\omega}(j,\cdo ,\iota _2,\iota _3)
e^{-r_1(|l_0-j_0|+|j_2|+|\lambda _3-\iota _3|)}
\vartheta (j,\iota _2,\lambda _3)^\rho}{\ell ^{q_1/\rho}}
\\[1ex]
\lesssim
\sum _{(j,\iota _2,\iota _3)\in \Lambda _1}
\left (
\nm {F _\omega (j,\cdo ,\iota _2,\iota _3)^\rho *
e^{-r_1|\cdo |}}{\ell ^{q_1/\rho}}
\right )
e^{-r_1(|l_0-j_0|+|j_2|+|\lambda _3-\iota _3|)}
\vartheta (j,\iota _2,\lambda _3)^\rho
\\[1ex]
\lesssim
\sum _{(j,\iota _2,\iota _3)\in \Lambda _1}
F _{q_1,\omega} (j,\iota _2,\iota _3)^\rho
e^{-r_1(|l_0-j_0|+|j_2|+|\lambda _3-\iota _3|)}
\vartheta (j,\iota _2,\lambda _3)^\rho ,
\end{multline*}
for every $r_1>0$.
By using H{\"o}lder's inequality with respect
to the $\iota _2$ variable in the sum
we obtain
\begin{multline*}
\nm {F _{0,\omega _0}(l_0,\cdo ,\lambda _3)
^\rho}{\ell ^{q_1/\rho}}
\\[1ex]
\lesssim
\sum _{(j,\iota _3)\in \Lambda _1}
\nm {F _{q_1,\omega} (j,\cdo ,\iota _3)^\rho}
{\ell ^{q_2/\rho}}
e^{-r_1(|l_0-j_0|+|j_2|+|\lambda _3-\iota _3|)}
\nm {\vartheta (j,\cdo ,\lambda _3)^\rho}{\ell ^{(q_2/\rho )'}} .
\\[1ex]
\le
C_\vartheta 
\sum _{(j,\iota _3)\in \Lambda _1}
\nm {F _{q_1,\omega} (j,\cdo ,\iota _3)}
{\ell ^{q_2}}^\rho
e^{-r_1(|l_0-j_0|+|j_2|+|\lambda _3-\iota _3|)}.
\end{multline*}
for every $r_1>0$. The remaining part of the proof
is performed in similar ways as in
the proof of Theorem \ref{Thm:TraceMod},
by applying Minkowski's and Young's inequalities
in suitable ways.
The details are left for the reader.
\end{proof}

\par

By letting $\omega$ and $\omega _0$ be related as
\begin{equation}\label{Eq:CondTraceWeight1B}
\omega _0(x_1,x_3,\xi _1,\xi _3)
\lesssim
\omega (x_1,x_2,x_3,\xi _1,\xi _2,\xi _3 )\vartheta (\xi _2 ,\xi _3),
\end{equation}
we get the following special case of Theorem \ref{Thm:TraceMod}.
The details are left for the reader.

\par

\begin{prop}\label{Prop:TraceMod}
Let $d_j$ be integers, $z\in \rr {d_2}$ be fixed,
$p$, $p_0$, $q$, $q_0$ and $r$
be the same as in Theorem
\ref{Thm:TraceMod}, and suppose that
$\omega \in \mascP _E(\rr {2(d_1+d_2+d_3)})$,
$\omega _0 \in \mascP _E(\rr {2(d_1+d_3)})$ and
$\vartheta \in \mascP _E(\rr {d_2+d_3})$
satisfy \eqref{Eq:CondTraceWeightModif} and
\eqref{Eq:CondTraceWeight1B}. 
Then $\Trm _{z}$ is continuous
from $M^{p,q}_{(\omega )}(\rr {d_1+d_2+d_3})$ to
$M^{p_0,q_0}_{(\omega _0)}(\rr {d_1+d_3})$, and
from $W^{p,q}_{(\omega )}(\rr {d_1+d_2+d_3})$ to
$W^{p_0,q_0}_{(\omega _0)}(\rr {d_1+d_3})$.
\end{prop}



\par

\begin{example}
An important case of Proposition \ref{Prop:TraceMod}
appears by letting $\vartheta$ in \eqref{Eq:CondTraceWeight1B}
be given by
$$
\vartheta (\xi _2,\xi _3) = \eabs {(\xi _2,\xi _3)}^{-s}\eabs {\xi _3}^{s_0}.
$$
Here
$$
s\ge \frac {d_2}{r},\quad s_0\le s-\frac {d_2}{r}.
$$
with the first inequality strict when $r<\infty$. Note 
that \eqref{Eq:CondTraceWeightModif} is fulfilled
for such $\vartheta$. In fact, for $r<\infty$ we have
\begin{align*}
\nm {\vartheta (\cdo ,\xi _3)}{L^{r}}^{r}
&=
\int _{\rr {d_2}}\eabs {(\xi _2,\xi _3)}^{-sr}
\eabs {\xi _3}^{s_0r}
\, d\xi _2
\\[1ex]
&=
\eabs {\xi _3}^{(s_0-s)r}
\int _{\rr {d_2}}\eabs {\xi _2/\eabs {\xi _3}}^{-sr}
\, d\xi _2
\\[1ex]
&=
\eabs {\xi _3}^{(s_0-s)r+d_2}
\int _{\rr {d_2}}\eabs {\xi _2}^{-sr}
\, d\xi _2 \lesssim 1,
\end{align*}
where the inequality follows from the assumptions on $s$ and $s_0$. By
similar arguments one gets
\eqref{Eq:CondTraceWeightModif} in the case when $r=\infty$.
The details are left for the reader.

\par

By choosing $\omega _0(x_0,\xi _0)= \eabs {\xi _3}^{s_0}$,
$\omega (x,\xi )= \eabs {(\xi _2,\xi _3)}^{s}$ and
$p=q=2$, we regain Sobolev's embedding theorem (see \eqref{Eq:TraceSobolev}
in the introduction).
\end{example}

\par

By letting
$$
d_j=d,\quad p_{j,k}=p_j \in (0,\infty ]
\quad \text{and}\quad
q_{j,k}=q_j \in (0,\infty ],\quad j\in \{ 1,2,3\} ,\ k\in \{1,\dots ,d \} ,
$$
and replacing the assumption \eqref{Eq:CondTraceWeight1B}
by
\begin{equation}\label{Eq:varthetaSimple}
\omega _0(x_1,x_3,\xi _1,\xi _3)
\asymp
\omega (x_1,x_2,x_3,\xi _1,\xi _2,\xi _3 )\eabs {\xi _2}^{-\theta},
\end{equation}
Theorems \ref{Thm:TraceMod} and \ref{Thm:TraceMod2}
give the following.

\par

\begin{prop}\label{Prop:TraceMod2}
Let $z\in \rr d$ be fixed, $p=(p_1,p_2,p_3)\in
(0,\infty ]^3$, $q=(q_1,q_2,q_3)\in (0,\infty ]^3$
and $r\in (0,\infty ]$ be such that
\eqref{Eq:CondTraceLebr1} holds,
and suppose that $\omega \in \mascP _E(\rr {6d})$ and
$\omega _0 \in \mascP _E(\rr {4d})$
satisfy \eqref{Eq:varthetaSimple},
for some
\begin{equation}\label{Eq:PropTraceMod2}
\theta \ge d
\left (
\max 
\left (
\frac 1{p_1},\frac 1{p_3},\frac 1{q_1},\frac 1{q_2},
\frac 1{q_3}
\right ) -\frac 1{q_2}
\right )
\end{equation}
with strict inequality when
$$
\min (p_1,p_3,q_1,q_3,1)<q_2.
$$
Then $\Trm _{z}$ is continuous and surjective
$M^{p,q}_{(\omega )}(\rr {3d})$ to
$M^{p_0,q_0}_{(\omega _0)}(\rr {2d})$, and
from $W^{p,q}_{(\omega )}(\rr {3d})$ to
$W^{p_0,q_0}_{(\omega _0)}(\rr {2d})$.
\end{prop}

\par

By choosing $p_1=p_2=p_3$ and $q_1=q_2=q_3$ in the previous
proposition, we get the following.

\par

\begin{cor}\label{Cor:TraceMod}
Let $z\in \rr d$ be fixed, $p,q,r\in  (0,\infty ]$
be such that \eqref{Eq:CondTraceLebr1} holds,
and suppose that $\omega \in \mascP _E(\rr {6d})$ and
$\omega _0 \in \mascP _E(\rr {4d})$
satisfy \eqref{Eq:varthetaSimple},
for some
\begin{equation}\label{Eq:CorTraceMod}
\theta \ge d
\left (
\max 
\left (
\frac 1p,\frac 1q
\right ) -\frac 1q
\right )
\end{equation}
with strict inequality when $\min (p,1)<q$.
%
%
Then $\Trm _{z}$ is continuous and surjective
$M^{p,q}_{(\omega )}(\rr {3d})$ to
$M^{p,q}_{(\omega _0)}(\rr {2d})$, and
from $W^{p,q}_{(\omega )}(\rr {3d})$ to
$W^{p,q}_{(\omega _0)}(\rr {2d})$.
\end{cor}

\par

\begin{rem} \label{Rem:TraceReaction}
Let $p,q,r\in (0,\infty )$ be such that $0<p<q$, $r=\min (1,p,q)$,
\begin{alignat*}{2}
p_1 &= (p,\dots ,p)\in (0,\infty ]^d, &
\quad
p_2 &=(p,\dots ,p)\in (0,\infty ]^{d-1},
\\[1ex]
q_1&=(q,\dots ,q,r)\in (0,\infty ]^d &
\quad \text{and}\quad
q_2 &= (q,\dots ,q)\in (0,\infty ]^{d-1}.
\end{alignat*}
Also let $s\in \mathbf R$,
$$
\omega _1(x_1,\xi _1) = \eabs {\xi _1}^s
\quad \text{and}\quad
\omega _2(x_2,\xi _2) = \eabs {\xi _2}^s,
\quad
x_1,\xi _1 \in \rr d,\ 
x_2,\xi _2 \in \rr {d-1}.
$$
Then it is proved in \cite[Theorem 3.1]{FeHuWa} that the trace map which takes
$$
(x_1,\dots ,x_d)\mapsto f(x_1,\dots ,x_d)
\quad \text{into}\quad
(x_1,\dots ,x_{d-1})\mapsto f(x_1,\dots ,x_{d-1},0)
$$
is continuous from $M^{p_1,q_1}_{(\omega _1)}(\rr d)$ to
$M^{p_2,q_2}_{(\omega _2)}(\rr {d-1})$.

\par

See also \cite[Theorem 2]{CunTsu} for similar results of trace 
operators on
Wiener amalgam type spaces $W^{p,q}_{(\omega )}(\rr d)$.

\par

We note that Theorem \ref{Thm:TraceMod} is more general compared
to the above result in the sense of relaxing the conditions of
the involved
weight functions, as well as including more exponents in the 
definition of
modulation spaces. For example, we notice that the trace result
above only deals with polynomial type weights of special kinds, and
the range of the Lebesgue exponents $p$ and $q$ is $(0,\infty )$
instead of the full interval $(0,\infty ]$, which includes
$\infty$.  

\par

On the other hand, \cite[Theorem 3.1]{FeHuWa} is a special case
of \cite[Theorem 3.3]{FeHuWa}, which concern trace results for
broader classes of $\alpha$-modulation spaces. Evidently,
Theorem \ref{Thm:TraceMod}, Propositions \ref{Prop:TraceMod},
\ref{Prop:TraceMod2} and Corollary \ref{Cor:TraceMod} do not show
any properties for $\alpha$-modulation spaces which are not 
modulation
spaces. We note however that some arguments in \cite{Schn}
or in the proof of
Theorem \ref{Thm:TraceMod}, might be suitable when extending
Theorems 3.1 and 3.3 in \cite{FeHuWa} to allow the Lebesgue exponents
to belong to the full interval $(0,\infty ]$.
(Cf. \cite[Remark 3.1]{Schn} and
\eqref{Eq:commdiagram}.)
\end{rem}

\par

\begin{rem}\label{Rem:TraceMod}
Let $p$, $q$, $\theta$, $\omega _0$, $\omega$,
$z$ be the same as in Proposition \ref{Prop:TraceMod2}.
It is clear that Proposition \ref{Prop:TraceMod2} is more general
compared to Corollary \ref{Cor:TraceMod}. On the other hand,
the involved
modulation space $M^{p,\infty ,p,q,q,q}_{(\omega )}(\rr {3d})$
in the domain
of $\Trm _z$ in Proposition \ref{Prop:TraceMod2} is in some sense
more complicated compared to the corresponding space
$M^{p,q}_{(\omega )}(\rr {3d})$ in Corollary
\ref{Cor:TraceMod}. An interesting 
question concerns whether Proposition \ref{Prop:TraceMod2}
involve important situations which are not included in
Corollary \ref{Cor:TraceMod}. 

\par

In this respect, on one hand we have
\begin{equation}\label{Eq:ModExtOp}
\begin{aligned}
f\in M^{p,\infty ,p,q,q,q}_{(\omega )}(\rr {3d})
\setminus
M^{p,q}_{(\omega )}(\rr {3d})
\quad \text{when}\quad
f(x,y,\zeta ) &= g(x,\zeta ), \quad x,y,\zeta \in \rr d,
\\[1ex]
g &\in M^{p,q}_{(\omega _0)}(\rr {2d})\setminus 0,
\end{aligned}
\end{equation}
if in addition $p<\infty$. Here $\omega$ and $\omega _0$
should satisfy \eqref{Eq:varthetaSimple} and
\eqref{Eq:CorTraceMod}. Hence we may identify
$M^{p,q}_{(\omega _0)}(\rr {2d})$ with a subset of
$M^{p,\infty ,p,q,q,q}_{(\omega )}(\rr {3d})$, but
not as a subset of $M^{p,q}_{(\omega )}(\rr {3d})$.
By straight-forward
computations it also follows that $f$ and $g$ in \eqref{Eq:ModExtOp}
fulfill $\Trm _zf = g$ for every $z\in \rr d$.

\par

On the other hand, if $p,p_0\in (0,\infty ]$, then Proposition
\ref{Prop:TraceMod2} shows that
\begin{alignat}{2}
\Trm _z \, &: \, & M^{p,p_0,p,q,q,q}_{(\omega )}(\rr {3d})
&\to M^{p,q}_{(\omega _0)}(\rr {2d})
\label{Eq:TraceMapProp}
\end{alignat}
is surjective.
\end{rem}

\par

\section{Continuity for pseudo-differential operators of
amplitude type on modulation spaces}\label{sec3}

\par

In this section we shall apply Proposition \ref{Prop:TraceMod}
to deduce identification properties between pseudo-differential operators
of amplitude types and standard types, when the symbol and
amplitude classes are modulation spaces. Then we combine
this with suitable continuity and compactness results in \cite{Toft19,Toft20}
to deduce continuity, Schatten-von Neumann and nuclearity properties
for pseudo-differential operators of amplitude types.

\par

\subsection{Pseudo-differential operators with amplitudes in modulation
spaces}

\par

It is suitable to consider a modified version of the modulation space
$M^{p,q}_{(\omega )}(\rr d)$ in order to formulate the main result. Let
$$
p_j,q_j\in (0,\infty ]^d,\ j\in \{ 1,2,3\} ,
\quad
p=(p_1,p_2,p_3)\in (0,\infty ]^{3d},
\quad
q=(q_1,q_2,q_3)\in (0,\infty ]^{3d},
$$
$r=\min (1,p,q)$, $\omega \in \mascP _E(\rr {6d})$ and
$\phi \in \Sigma _1(\rr {3d})$. Then let
\begin{alignat*}{2}
\nm F{\maclL ^{p,q}(\rr {6d})} &\equiv \nm {G_1}{L^{p,q}(\rr {6d})},
& \quad
G_1(x,y,\zeta ,\xi ,\eta ,z) &= F(x,x+y,\zeta ,\xi ,\eta ,z),
\intertext{and}
\nm F{\maclL ^{p,q}_*(\rr {6d})} &\equiv \nm {G_2}{L^{q,p}(\rr {6d})},
& \quad
G_2(\xi ,\eta ,z,x,y,\zeta ) &= F(x,x+y,\zeta ,\xi ,\eta ,z),
\end{alignat*}
when $F\in L^r_{loc}(\rr {6d})$. The modulation spaces
$$
\maclM ^{p,q}_{(\omega )}(\rr {3d})
=
\maclM ^{p_1,p_2,p_3,q_1,q_2,q_3}_{(\omega )}(\rr {3d})
\quad \text{and}\quad
\maclW ^{p,q}_{(\omega )}(\rr {3d})
=
\maclW ^{p_1,p_2,p_3,q_1,q_2,q_3}_{(\omega )}(\rr {3d})
$$
are the sets of all $a\in \Sigma _1'(\rr {3d})$ such that
$$
\nm a{\maclM ^{p,q}_{(\omega )}}
=
\nm a{\maclM ^{p_1,p_2,p_3,q_1,q_2,q_3}_{(\omega )}}
\equiv
\nm {V_\phi a\cdot \omega}{\maclL ^{p,q}}
$$
respectively
$$
\nm a{\maclW ^{p,q}_{(\omega )}}
=
\nm a{\maclW ^{p_1,p_2,p_3,q_1,q_2,q_3}_{(\omega )}}
\equiv
\nm {V_\phi a\cdot \omega}{\maclL ^{p,q}_*}
$$
is finite.
The spaces 
$\maclM ^{p,q}_{(\omega )}(\rr {3d})$ and 
$\maclM ^{p,q}_{(\omega )}(\rr {3d})$ are equipped 
with the topologies supplied by the quasi-norms
$\nm \cdo{\maclM ^{p,q}_{(\omega )}}$
and
$\nm \cdo{\maclW ^{p,q}_{(\omega )}}$, respectively.

\par

As for the classical modulation spaces $M ^{p,q}_{(\omega )}(\rr {3d})$,
we put
\begin{align*}
\maclM ^{p_{0,1},p_{0,2},p_{0,3},q_{0,1},q_{0,2},q_{0,3}}_{(\omega )}
&=
\maclM ^{p_1,p_2,p_3,q_1,q_2,q_3}_{(\omega )}
\intertext{and}
\maclW ^{p_{0,1},p_{0,2},p_{0,3},q_{0,1},q_{0,2},q_{0,3}}_{(\omega )}
&=
\maclW ^{p_1,p_2,p_3,q_1,q_2,q_3}_{(\omega )}
\end{align*}
when $p_{0,j},q_{0,j}\in (0,\infty ]$,
$$
p_j=(p_{0,j},\dots ,p_{0,j})\in (0,\infty ]^d
\quad \text{and}\quad
q_j=(p_{0,j},\dots ,q_{0,j})\in (0,\infty ]^d,
\quad
j=1,2,3.
$$

\par

The following lemma justifies the introduction of
$\maclM ^{p,q}_{(\omega )}(\rr {3d})$.

\par

\begin{lemma}\label{Lemma:IdentTranslMod}
Let $p_j,q_j\in (0,\infty ]^d$, $j=1,2,3$, $p=(p_1,p_2,p_3)$,
$q=(q_1,q_2,q_3)$,
$\omega \in \mascP _E(\rr {6d})$,
$$
\omega _0(x,y,\zeta ,\xi ,\eta ,z)
=
\omega (x,x+y,\zeta ,\xi -\eta ,\eta ,z)
$$
and let $T_1$ be the map on $\Sigma _1'(\rr {3d})$, given by
$$
(T_1a)(x,y,\zeta ) = a(x,x+y,\zeta ),\qquad a\in \Sigma _1'(\rr {3d}).
$$
Then $T_1$ restricts to a homeomorphism from
$\maclM ^{p,q}_{(\omega )}(\rr {3d})$
to
$M^{p,q}_{(\omega _0)}(\rr {3d})$
\end{lemma}

\par

\begin{proof}
By straight-forward computations we get
$$
|(V_{T\phi }(Ta))(x,y,\zeta ,\xi ,\eta ,z)|
=
|V_{\phi}a(x,x+y,\zeta ,\xi -\eta ,\eta ,z)|
$$
when $a\in \Sigma _1'(\rr {3d})$ and
$\phi \in \Sigma _1(\rr {3d})$.
The result now follows by multiplying the equality with
$\omega _0$ and then apply the $L^{p,q}$ quasi-norm.
\end{proof}

\par

We also need the following lemma. We omit the proof since the
result follows from \cite[Proposition 2.8]{Toft15} and its proof.

\par

\begin{lemma}\label{Lemma:ExpOpModAmpl}
Let $p$, $q$, $\omega$ be the same as in Lemma
\ref{Lemma:IdentTranslMod}, let
$$
\omega _0(x,y,\zeta ,\xi ,\eta ,z)
=
\omega (x,y+z,\eta +\zeta ,\xi ,\eta ,z),
$$
and let $T_2$ be the map
$$
(T_2a)(x,y,\zeta ,\xi ,\eta ,z)
=
(e^{i\scal {D_\xi}{D_y}}a)(x,y,\zeta ,\xi ,\eta ,z),
\qquad a\in \Sigma _1'(\rr {3d}).
$$
Then $T_2$ from $\Sigma _1'(\rr {3d})$ to $\Sigma _1'(\rr {3d})$
restricts to a homeomorphism from
$M^{p,q}_{(\omega )}(\rr {3d})$ to
$M^{p,q}_{(\omega _0)}(\rr {3d})$.
\end{lemma}

\par

Let $a\in \Sigma _1(\rr {3d})$, $T_1$ and $T_2$ be as in Lemmas
\ref{Lemma:IdentTranslMod} and \ref{Lemma:ExpOpModAmpl}.
Then \eqref{Eq:aToa0B} shows that
\begin{equation}\tag*{(\ref{Eq:aToa0B})$'$}
a_0 = (\Trm _0 \circ T_2 \circ T_1)a.
\end{equation}
For any $a\in \Sigma _1'(\rr {3d})$ such that the right-hand side of
\eqref{Eq:aToa0B}$'$ makes sense as an element
in $\Sigma _1'(\rr {2d})$, we define $\op (a)$ as the operator $\op _0(a)$.
By previous investigations  we get the following.
Here the involved weights should satisfy conditions of the form
\begin{equation}\label{Eq:CondTraceWeightOp}
\omega _0(x,\zeta ,\xi ,z)
\lesssim
\omega (x,x+z,\zeta +\eta ,\xi -\eta ,\eta ,z)\vartheta (\eta ,z)
\end{equation}
and
\begin{equation}\label{Eq:CondTraceWeight1AOp}
\omega (x,y,\zeta ,\xi ,\eta ,z)
\lesssim
\omega _0(x,\zeta ,\xi ,z)e^{r_0(|y|+|\eta |)}.
\end{equation}

\par

\begin{thm}\label{Thm:AmpOpPsOpMod}
Let $p_j,q_j,r\in (0,\infty ]^{d}$, $j=1,2,3$,
$\omega \in \mascP _E(\rr {6d})$,
$\omega _0 \in \mascP _E(\rr {4d})$, $p$, $p_0$, $q$ and $q_0$
be such that \eqref{Eq:CondTraceLeb},
\eqref{Eq:CondTraceLebr1}
and \eqref{Eq:CondTraceWeightOp} hold true for some $r_0\ge 0$ and
$\vartheta \in \mascP _E (\rr {2d})$ such that
\eqref{Eq:CondTraceWeightModif} holds.
Then the following is true:
\begin{enumerate}
\item if $a\in \maclM ^{p,q}_{(\omega )}(\rr {3d})$, then
there is a unique $a_0\in M^{p_0,q_0}_{(\omega _0)}(\rr {2d})$
such that $\op (a) = \op _0(a_0)$;

\vrum

\item if in addition \eqref{Eq:CondTraceWeight1AOp} holds true,
and $a_0\in M^{p_0,q_0}_{(\omega _0)}(\rr {2d})$,
then there is an $a\in \maclM ^{p,q}_{(\omega )}(\rr {3d})$
such that $\op (a) = \op _0(a_0)$.
\end{enumerate}
\end{thm}

\par

\begin{proof}
The assertion (1) follows by combining Theorem
\ref{Thm:TraceMod} (1), Lemma \ref{Lemma:IdentTranslMod}
and Lemma \ref{Lemma:ExpOpModAmpl}, and the observation
that \eqref{Eq:CondTraceWeightOp} is the same as that
\begin{equation}\tag*{(\ref{Eq:CondTraceWeightOp})$'$}
\omega _0(x,\zeta ,\xi ,z)
\lesssim
\omega (x,x+y+z,\zeta +\eta ,\xi -\eta ,\eta ,z)e^{r_0|y|}
\vartheta (\eta ,z)
\end{equation}
should hold for some $r_0>0$, due to the moderateness of $\omega$.
The details are left for the reader.

\par

The assertion (2) follows by letting
$$
a(x,y,\zeta )= e^{-i\scal {D_\xi}{D_y}}(a_0(x,\zeta )\fy (y-x)),
$$
where $\fy \in \Sigma _1(\rr d)$ fulfills $\fy (0)=1$, and
applying Theorem \ref{Thm:TraceMod} (2), Lemma
\ref{Lemma:IdentTranslMod} and Lemma
\ref{Lemma:ExpOpModAmpl}. The details are left for the reader.
\end{proof}

\par

In a similar way as when passing from
Theorem \ref{Thm:TraceMod} into Proposition \ref{Prop:TraceMod},
it follows that the following result is a special cases of
Theorem \ref{Thm:AmpOpPsOpMod}.
The details are left for the reader. Here the involved weights
are related as
\begin{equation}\label{Eq:CondTraceWeight1BOp}
\omega _0(x,\zeta ,\xi ,z)
\asymp
\omega (x,x+z,\zeta +\eta ,\xi -\eta ,\eta ,z)\vartheta (\eta ,z)
\end{equation}

\par

\begin{prop}\label{Prop:AmpOpPsOpMod}
Let $p$, $p_0$, $q$, $q_0$, $r$
and $\vartheta$ be the same as in Theorem
\ref{Thm:AmpOpPsOpMod}, and suppose that
$\omega \in \mascP _E(\rr {6d})$ and
$\omega _0 \in \mascP _E(\rr {4d})$
satisfy \eqref{Eq:CondTraceWeight1BOp}. Then
$$
\sets {\op (a)}{a\in \maclM ^{p,q}_{(\omega )}(\rr {3d})}
=
\sets {\op _0(a_0)}{a_0\in M ^{p,q}_{(\omega )}(\rr {2d})}.
$$
\end{prop}

\par

In the same way it follows that the next result is a special
case of the previous one. The details are left for the reader.
Here the relationship between
the weights are more specified into
\begin{equation}\tag*{(\ref{Eq:CondTraceWeight1BOp})$'$}
\omega _0(x,\zeta ,\xi ,z)
\asymp
\omega (x,x+z,\zeta +\eta ,\xi -\eta ,\eta ,z)\eabs \eta ^{-\theta}
\end{equation}

\par

\begin{prop}\label{Prop:AmpOpPsOpMod2}
Let $p,q\in (0,\infty ]$, and suppose that
$\omega \in \mascP _E(\rr {6d})$ and
$\omega _0 \in \mascP _E(\rr {4d})$
satisfy \eqref{Eq:CondTraceWeight1BOp}$'$,
for some $\theta$ such that \eqref{Eq:CorTraceMod}
holds with strict inequality when $q> \min (p,1)$.
Then
$$
\sets {\op (a)}{a\in \maclM ^{p,q}_{(\omega )}(\rr {3d})}
=
\sets {\op _0(a_0)}{a_0\in M ^{p,q}_{(\omega _0)}(\rr {2d})}.
$$
\end{prop}

\par

We observe that
$\maclM ^{p,q}_{(\omega )}(\rr {3d})=M ^{p,q}_{(\omega )}(\rr {3d})$
in the last proposition.

\par

We may combine the previous results with other established results for
pseudo-differential operators when acting on modulation spaces. For example,
the following result is a straight-forward consequence of
Theorem \ref{Thm:AmpOpPsOpMod} or
Proposition \ref{Prop:AmpOpPsOpMod} with Theorem \ref{Thm:OpCont}.
The details are left for the reader. Here the condition
\eqref{Eq:WeightPseudoCond} from the introduction
needs to be modified into
\begin{alignat}{2}
\frac {\omega _2(x,\xi )}{\omega _1(z,\zeta )}
&\lesssim
\omega (x,z,\zeta +\eta ,\xi -\zeta -\eta ,\eta ,z-x)\vartheta (\eta ,z), &
\qquad x,z,\xi ,\eta ,\zeta  &\in \rr d.
\tag*{(\ref{Eq:WeightPseudoCond})$'$}
\end{alignat}
Recall also Subsection \ref{subsec1.4} for ordering
relations between elements in $(0,\infty ]$ and
elements in $(0,\infty ]^d$.

\par

\begin{thm}
Let $\sigma \in \operatorname{S}_{2d}$,
$p,q,r\in (0,\infty ]$ and $p_1,p_2\in (0,\infty ]^{2d}$
be such that 
$$
\frac 1{p_2} - \frac 1{p_1}
=
\frac 1p +\min \left ( 0,\frac 1q -1 \right ),
\quad q\le \min (p_2)\le \max(p_2)\le p,
$$
and
\begin{equation}\label{Eq:MainThmAddLebCond}
\max \left ( 1,\frac 1p,\frac 1q  \right ) -\frac 1q \le \frac 1r,
\end{equation}
hold with latter inequality strict when $q>\min (1,p)$. Also let
$\omega \in \mascP _E(\rr {6d})$,
$\omega _1,\omega _2\in \mascP _E(\rr {2d})$
and 
$\vartheta \in \mascP _E (\rr {2d})$ be such that
\eqref{Eq:WeightPseudoCond}$'$ and
\eqref{Eq:CondTraceWeightModif} hold.
If $a\in \maclM ^{p,\infty ,p,q,q,q}_{(\omega )}(\rr {3d})$,
then $\op (a)$ from
$\Sigma _1(\rr d)$ to $\Sigma _1'(\rr d)$ is
uniquely extendable to a continuous
map from $M^{p_1}_{\sigma ,(\omega _1)}(\rr d)$ to
$M^{p_2}_{\sigma ,(\omega _2)}(\rr d)$,
and
\begin{equation}
\begin{aligned}
\nm {\op (a)f}{M^{p_2}_{\sigma ,(\omega _2)}}
&\lesssim
\nm a{\maclM ^{p,\infty ,p,q,q,q}_{(\omega )}}
\nm f{M^{p_1}_{\sigma ,(\omega _1)}},
\\[1ex]
a &\in \maclM ^{p,\infty ,p,q,q,q}_{(\omega )}(\rr {3d}),
\ 
f\in M^{p_1}_{\sigma ,(\omega _1)}(\rr d).
\end{aligned}
\end{equation}
\end{thm}

\par

As a special case of the previous result we have the
following extension of Theorem \ref{Thm:AmpOpModCont}
in the introduction.

\par

\renewcommand{\rubrik}{Theorem \ref{Thm:AmpOpModCont}$'$\!}

\par

\begin{tom}
Let $p,q,p_j,q_j,r\in (0,\infty ]$, $j=1,2$, be such that
$$
\frac 1{p_2} - \frac 1{p_1} = \frac 1{q_2} -\frac 1{q_1}
=
\frac 1p +\min \left ( 0,\frac 1q -1 \right ),
\quad q\le p_2,q_2\le p,
$$
and \eqref{Eq:MainThmAddLebCond} holds
with strict inequality when $q>\min (1,p)$. Also let
$\omega \in \mascP _E(\rr {6d})$,
$\omega _1,\omega _2\in \mascP _E(\rr {2d})$
and 
$\vartheta \in \mascP _E (\rr {2d})$ be such that
\eqref{Eq:WeightPseudoCond}$'$ and \eqref{Eq:CondTraceWeightModif} hold.
If $a\in \maclM ^{p,\infty ,p,q,q,q}_{(\omega )}(\rr {3d})$,
then $\op (a)$ is continuous
from $M^{p_1,q_1}_{(\omega _1)}(\rr d)$ to
$M^{p_2,q_2}_{(\omega _2)}(\rr d)$, and from
$W^{p_1,q_1}_{(\omega _1)}(\rr d)$ to $W^{p_2,q_2}_{(\omega _2)}(\rr d)$.
%
\end{tom}

\par

We observe that if $p=\infty$ and $q\in (0,1]$ in Theorem
\ref{Thm:AmpOpModCont}$'$, then $p_1=p_2\in [q,\infty ]$,
$q_1=q_2\in [q,\infty ]$ and we may choose $r=\infty$ and thereby
choose $\vartheta (\eta ,z)=1$ everywhere. It is now evident that
Theorem \ref{Thm:AmpOpModCont}$'$ takes the form of
Theorem \ref{Thm:AmpOpModCont} in the introduction for such
choices of $p$ and $q$.

\par

By combining Theorem \ref{Thm:AmpOpPsOpMod} or
Proposition \ref{Prop:AmpOpPsOpMod} with \eqref{Eq:OpContSpec1Modif}
instead of Theorem \ref{Thm:OpCont} we get the following. The details
are left for the reader.

\par

\begin{thm}\label{Thm:PseudoSymbWienerAm}
Let
$p,q,r\in [1,\infty ]$
be such that 
and
\begin{equation}\label{Eq:MainThmAddLebCondAmal}
\frac 1{q'} \le \frac 1r,
\end{equation}
hold. Also let
$\omega \in \mascP _E(\rr {6d})$,
$\omega _1,\omega _2\in \mascP _E(\rr {2d})$
and 
$\vartheta \in \mascP _E (\rr {2d})$ be such that
\eqref{Eq:WeightPseudoCond}$'$ and
\eqref{Eq:CondTraceWeightModif} hold.
If $a\in \maclW ^{p,\infty ,p,q,q,q}_{(\omega )}(\rr {3d})$,
then $\op (a)$ from
$\Sigma _1(\rr d)$ to $\Sigma _1'(\rr d)$ is
uniquely extendable to a continuous
map from $M^{q',p'}_{(\omega _1)}(\rr d)$ to
$W^{p,q}_{(\omega _2)}(\rr d)$,
and
\begin{equation}
\begin{aligned}
\nm {\op (a)f}{W^{p,q}_{(\omega _2)}}
&\lesssim
\nm a{\maclW ^{p,\infty ,p,q,q,q}_{(\omega )}}
\nm f{M^{q',p'}_{(\omega _1)}},
\\[1ex]
a &\in \maclW ^{p,\infty ,p,q,q,q}_{(\omega )}(\rr {3d}),
\ 
f\in M^{q',p'}_{\sigma ,(\omega _1)}(\rr d).
\end{aligned}
\end{equation}
\end{thm}

\par

\subsection{Pseudo-differential operators with amplitudes in
weighted Gevrey analogies of the H{\"o}rmander class
$S_{0,0}^0$}

\par

We shall next apply Theorem \ref{Thm:AmpOpModCont} to deduce
some continuity properties for amplitude type pseudo-differential
operators with certain smooth symbols when acting on modulation
spaces. For any $\omega \in \mascP _E(\rr d)$ and $s\ge 1$,
let
\begin{align}
\nm f{S^{(\omega )}_{s;h}}
&\equiv
\sup _{\alpha \in \nn d}
\left (
\frac {\nm {(\partial ^\alpha f)/\omega}{L^\infty} }{h^{|\alpha |\alpha !^s}}
\right )
\\[1ex]
S^{(\omega )}(\rr d)
&=
\sets {f\in C^\infty (\rr d)}
{\nm {(\partial ^\alpha f)/\omega}{L^\infty} <\infty ,\
\text{for every}\ \alpha \in \nn d}
\\[1ex]
S^{(\omega )}_s(\rr d)
&=
\sets {f\in C^\infty (\rr d)}{\nm f{S^{(\omega )}_{s;h}} <\infty ,\
\text{for some}\ h>0}
\intertext{and}
S^{(\omega )}_{0,s}(\rr d)
&=
\sets {f\in C^\infty (\rr d)}{\nm f{S^{(\omega )}_{s;h}} <\infty ,\
\text{for every}\ h>0}
\end{align}

\par

For $\omega \in \mascP _E(\rr {d})$ and $s\ge 1$, set
\begin{align}
\omega _{r} (x,\xi )
=
\omega (x)\eabs {\xi}^{-r}
\quad \text{and}\quad
\omega _{r,s} (x,\xi )
=
\omega (x) e^{-r|\xi |^{\frac 1s}}.
\label{Eq:omegaExpand1}
\end{align}

\par

The following lemma is a straight-forward consequence of
\cite[Proposition 2.5]{AbCaTo}, \cite[Proposition 2.7]{GroTo},
Proposition \ref{Prop:ModSpacesBasicProp}
(2) and the fact that
$$
M^{p,\infty}_{(\omega _2)}(\rr d) \subseteq M^{p,q}_{(\omega _1)}(\rr d)
\quad \text{when}\quad
q\in (0,\infty ],\ N>\frac dq,\ \omega _2(x,\xi ) = \omega _1(x,\xi )\eabs \xi ^N
$$
(see also \cite{CoNiRo}). The
details are left for the reader. Here recall Subsection \ref{subsec1.1}
for the definition of involved weight classes.

\par

\begin{lemma}\label{Lemma:SymbModDescr}
Let $q\in (0,\infty ]$, $s\ge 1$, $\omega \in \mascP _E(\rr d)$, and let
$\omega _{r}$ and $\omega _{r,s}$ be as in
\eqref{Eq:omegaExpand1}. 
Then the following is true:
\begin{enumerate}
\item if in addition $\omega \in \mascP (\rr d)$, then
$$
S^{(\omega )}(\rr d)
=
\underset{r>0}{\textstyle{\bigcap}} M^{\infty ,q}_{(1/\omega _{r})}(\rr d) \text ;
$$

\vrum

\item  if in addition $\omega \in \mascP _{E,s}^0(\rr d)$, then
$$
S^{(\omega )}_s(\rr d)
=
\underset{r>0}{\textstyle{\bigcup}} M^{\infty ,q}_{(1/\omega _{r,s})}(\rr d) \text ;
$$

\vrum

\item  if in addition $\omega \in \mascP _{E,s}(\rr d)$, then
$$
S^{(\omega )}_{0,s}(\rr d)
=
\underset{r>0}{\textstyle{\bigcap}} M^{\infty ,q}_{(1/\omega _{r,s})}(\rr d)
\text .
$$
\end{enumerate}
\end{lemma}

\par

We have now the following. Here the involved weights should satisfy
\begin{equation}\label{Eq:omegaCondSimpl}
\frac {\omega _1(y,\xi )}{\omega _2(x,\xi )}
\gtrsim
\omega _0(x,y,\xi ).
\end{equation}

\par

\begin{prop}
Let $p,q\in (0,\infty ]$ and $s\ge 1$.
Then the following is true:
\begin{enumerate}
\item if 
$\omega _0\in \mascP (\rr {3d})$,
$\omega _j\in \mascP (\rr {2d})$, $j=1,2$, satisfy
\eqref{Eq:omegaCondSimpl} and $a\in S^{(\omega _0)}(\rr {3d})$,
then $\op (a)$ is continuous from $M^{p,q}_{(\omega _1)}(\rr d)$
to $M^{p,q}_{(\omega _2)}(\rr d)$;

\vrum

\item if 
$\omega _0\in \mascP _{E,s}(\rr {3d})$,
$\omega _j\in \mascP _{E,s}(\rr {2d})$, $j=1,2$, satisfy
\eqref{Eq:omegaCondSimpl} and $a\in S^{(\omega _0)}_{0,s}(\rr {3d})$,
then $\op (a)$ is continuous from $M^{p,q}_{(\omega _1)}(\rr d)$
to $M^{p,q}_{(\omega _2)}(\rr d)$;

\vrum

\item if 
$\omega _0\in \mascP _{E,s}^0(\rr {3d})$,
$\omega _j\in \mascP _{E,s}^0(\rr {2d})$, $j=1,2$, satisfy
\eqref{Eq:omegaCondSimpl} and $a\in S^{(\omega _0)}_{s}(\rr {3d})$,
then $\op (a)$ is continuous from $M^{p,q}_{(\omega _1)}(\rr d)$
to $M^{p,q}_{(\omega _2)}(\rr d)$.
\end{enumerate}
\end{prop}

\par

\begin{proof}
We only prove (3). The other assertions follow by similar argument
and are left for the reader.

\par

Since all weights are moderate, \eqref{Eq:omegaCondSimpl} gives
$$
\frac {\omega _2(x,\xi )}{\omega _1(z,\zeta )}
\lesssim
\omega _0(x,z,\zeta )^{-1}e^{r|\xi -\zeta |^{\frac 1s}}
$$
for every $r>0$. Hence, using that
$$
e^{r|x+y|^{\frac 1s}} \le e^{r|x|^{\frac 1s}}e^{r|y|^{\frac 1s}}
$$
we obtain
$$
\frac {\omega _2(x,\xi )}{\omega _1(z,\zeta )}
\lesssim
\omega _0(x,z,\zeta +\eta )^{-1}
e^{r(|\xi -\zeta |^{\frac 1s} +|\xi -\zeta -\eta |^{\frac 1s} + |z-x|^{\frac 1s})}
$$
for every $r>0$. The result now follows by combining
Theorem \ref{Thm:AmpOpModCont} with Lemma
\ref{Lemma:SymbModDescr}.
\end{proof}

\par

\subsection{Schatten and Nuclearity properties for
pseudo-differential operators of amplitude type}

\par

Next we shall combine Theorem \ref{Thm:AmpOpPsOpMod}
with results in \cite{Toft20}
to find Schatten-von Neumann and nuclear properties for
pseudo-differential operators of amplitude type with symbols
in suitable modulation spaces.

\par

Let $p\in (0,\infty ]$, $\mascB _1$ and $\mascB _2$ be quasi-Banach spaces, and
let $T\in \maclL (\mascB _1,\mascB _2)$.
Then the singular number of $T$ of order $j\ge 1$ is defined by
$$
\sigma _j(T) \equiv \inf \nm {T-T_j}{\maclL(\mascB _1,\maclB _2)},
$$
where the infimum is taken over all $T_j\in \maclL (\mascB _1,\mascB _2)$
with rank at most $j-1$. Then $\mascI _p(\mascB _1,\mascB _2)$
consists of all $T\in \maclL (\mascB _1,\mascB _2)$ such that
$$
\nm T{\mascI _p(\mascB _1,\mascB _2)}
\equiv
\nm {\{ \sigma _j(T)\} _{j=1}^\infty}{\ell ^p(\mathbf Z_+)}
$$
is finite. We observe that $\mascI _\infty (\mascB _1,\mascB _2)
= \maclL (\mascB _1,\mascB _2)$ with the same quasi-norms.
We recall that if in addition $\mascB _1$ and $\mascB _2$ are
Hilbert spaces, then $\mascI _2 (\mascB _1,\mascB _2)$
and $\mascI _1 (\mascB _1,\mascB _2)$ are the sets of
Hilbert-Schmidt and trace class operators from $\mascB _1$
to $\mascB _2$, respectively.

\par


\par

The following result is now a straight-forward consequence
of \cite[Theorem 3.4]{Toft19} and Theorem
\ref{Thm:AmpOpPsOpMod}. The details are left for the reader.

\par

\renewcommand{\rubrik}{Theorem \ref{Thm:SchattenPseudo}$'$\!}

\par

\begin{tom} 
Let $p,q,r\in (0,\infty ]$ be such that $q\le \min (p,p')$
and \eqref{Eq:MainThmAddLebCond} hold,
$\omega \in \mascP _E(\rr {6d})$ and
$\omega _1,\omega _2\in \mascP _E(\rr {2d})$ be such that
\eqref{Eq:WeightPseudoCond}$'$ holds with $\vartheta$ being
the same as in Theorem \ref{Thm:AmpOpModCont}$'$.
If $a\in \maclM ^{p,\infty ,p,q,q,q}_{(\omega )}(\rr {3d})$, then
$\op (a)\in \mascI _p(M^2_{(\omega _1)}(\rr d),M^2_{(\omega _2)}(\rr d))$,
and
$$
\nm {\op (a)}{\mascI _p(M^2_{(\omega _1)},M^2_{(\omega _2)})}
\lesssim \nm a {\maclM ^{p,\infty ,p,q,q,q}_{(\omega )}},\quad
a\in \maclM ^{p,\infty ,p,q,q,q}_{(\omega )}(\rr {3d}).
$$
\end{tom}

\par

Next we perform similar discussions for nuclear operators.
Let $p\in (0,1]$, $\mascB _1$ be a Banach space with dual $\mascB _1'$,
$\mascB _2$ be a quasi-Banach space, and
let $T\in \maclL (\mascB _1,\mascB _2)$. Then $T$ is said to be
in the class $\mascN _p(\mascB _1,\mascB _2)$ of nuclear operators of
order $p$, if there are sequences $\{ \ep _j\} _{j=1}^\infty \subseteq \mascB _1'$
and $\{ e_j\} _{j=1}^\infty \subseteq \mascB _2$ such that
\begin{equation}\label{Eq:NuclearDef}
Tf = \sum _{j=1}^\infty \scal {\ep _j}f e_j
\quad \text{and}\quad
\sum _{j=1}^\infty (\nm {\ep _j}{\mascB _1'}\nm {e_j}{\mascB _2})^p
<\infty ,
\quad
f\in \mascB _1. 
\end{equation}
We also set
$$
\nm T{\mascN _p(\mascB _1,\mascB _2)}
\equiv
\inf \left ( \sum _{j=1}^\infty (\nm {\ep _j}{\mascB _1'}\nm {e_j}{\mascB _2})^p
 \right )^{\frac 1p},
$$
where the infimum is taken over all
$\{ \ep _j\} _{j=1}^\infty \subseteq \mascB _1'$
and $\{ e_j\} _{j=1}^\infty \subseteq \mascB _2$ such that
\eqref{Eq:NuclearDef} holds true. (See e.{\,}g.
\cite{DeRuTo,DeRuWa,Toft20}.)

\par

We observe that $\mascN _p(\mascB _1,\mascB _2)$
decreases with $\mascB _1$, and increases with
$p$ and $\mascB _2$. If in addition
$\mascB _1$ and $\mascB _2$ are Hilbert spaces, then
the spectral theorem implies that
$$
\mascN _p(\mascB _1,\mascB _2)
=
\mascI _p(\mascB _1,\mascB _2)
$$
with the same quasi-norms. (See e.{\,}g. \cite{Grot}.)
As a consequence we have
\begin{equation}\label{Eq:NuclearSchattenComp}
\mascN _p(M^\infty _{(\omega _1)}(\rr d),M^p_{(\omega _2)}(\rr d))
\subseteq
\mascN _p(M^2_{(\omega _1)}(\rr d),M^2_{(\omega _2)}(\rr d))
=
\mascI _p(M^2_{(\omega _1)}(\rr d),M^2_{(\omega _2)}(\rr d)).
\end{equation}

The following result is now a consequence
of \cite[Theorem 4.2]{Toft20} and Theorem
\ref{Thm:AmpOpPsOpMod}. The details are left for the reader.

\par

\renewcommand{\rubrik}{Theorem \ref{Thm:NuclPseudo}$'$\!}

\par

\begin{tom} 
Let $p\in (0,1]$,
$\omega \in \mascP _E(\rr {6d})$ and
$\omega _1,\omega _2\in \mascP _E(\rr {2d})$ be such that
\eqref{Eq:WeightPseudoCond} holds.
If $a\in \maclM ^{p,\infty ,p,p,p,p}_{(\omega )}(\rr {3d})$, then
$\op (a)\in \mascN _p(M^\infty _{(\omega _1)}(\rr d),M^p_{(\omega _2)}(\rr d))$,
and
$$
\nm {\op (a)}{\mascN _p(M^\infty _{(\omega _1)},M^p_{(\omega _2)})}
\lesssim \nm a {\maclM ^{p,\infty ,p,p,p,p}_{(\omega )}},\quad
a\in \maclM ^{p,\infty ,p,p,p,p}_{(\omega )}(\rr {3d}).
$$
\end{tom}

\par

Due to \eqref{Eq:NuclearSchattenComp} it is evident that
Theorem \ref{Thm:NuclPseudo}$'$ improves Theorem
\ref{Thm:SchattenPseudo}$'$ when $p\in (0,1]$.

\par

\begin{rem}
There are other nuclearity
properties for pseudo-differential operators of the types
$\op _A(a)$ with symbols in
modulation spaces, available in the literature and which are
comparable with those in \cite{Toft20}
(see e.{\,}g. \cite{DeRuTo,DeRuWa}).
In this context we have chosen to use the nuclearity
results in \cite{Toft20} because they seems to be sharp.
\end{rem}

\par

\end{document}